\documentclass[11pt,reqno]{amsart}

\usepackage[utf8]{inputenc}

\usepackage{amsmath,amsthm,amssymb}
\usepackage{amssymb}

\usepackage{graphics}
\usepackage{hyperref}
\usepackage[usenames, dvipsnames]{xcolor}

\usepackage{mathtools}
\mathtoolsset{showonlyrefs}

\usepackage[square,sort,comma,numbers]{natbib}
\usepackage{verbatim}

\definecolor{darkblue}{rgb}{0.0,0.0,0.3}
\hypersetup{colorlinks,breaklinks,
  linkcolor=darkblue,urlcolor=darkblue,
anchorcolor=darkblue,citecolor=darkblue}

\theoremstyle{plain}
\newtheorem{theorem}{Theorem}[section]
\newtheorem*{theorem*}{Theorem}
\newtheorem{lemma}[theorem]{Lemma}
\newtheorem{proposition}[theorem]{Proposition}
\newtheorem*{proposition*}{Proposition}

\newtheorem*{corollary*}{Corollary}

\theoremstyle{definition}
\newtheorem{remark}[theorem]{Remark}

\numberwithin{equation}{section}

\renewcommand{\Im}{\operatorname{Im}}
\renewcommand{\Re}{\operatorname{Re}}

\DeclareMathOperator{\GL}{GL}
\DeclareMathOperator{\sgn}{sgn}
\DeclareMathOperator*{\Res}{Res}
\DeclareMathOperator{\Sym}{Sym}

\title[Arithmetic Progressions of Squares]{Arithmetic Progressions of Squares
and Multiple Dirichlet Series}
\author[Hulse]{Thomas A. Hulse}
\author[Kuan]{Chan Ieong Kuan}
\author[Lowry-Duda]{David Lowry-Duda}
\author[Walker]{Alexander Walker}

\begin{document}

\maketitle

\begin{abstract}
We study a Dirichlet series in two variables which counts primitive three-term arithmetic
progressions of squares. We show that this multiple Dirichlet series has
meromorphic continuation to $\mathbb{C}^2$ and use Tauberian methods to obtain
counts for arithmetic progressions of squares and rational points on
$x^2+y^2=2$.
\end{abstract}

\section{Introduction}

In this paper, we produce estimates for the number of primitive three-term
arithmetic progressions of integer squares, $\{a^2, b^2, c^2\}$ with $c^2-b^2 =
b^2 -a^2$, whose terms are constrained to lie in certain regions.
As no nontrivial arithmetic progression of integer squares has more than three terms
--- stated by Fermat and proved by Euler (among others) ---
we refer to three-term arithmetic
progressions more succinctly as just arithmetic progressions, or APs.
(See~\cite[Vol II, Ch. XIV]{dickson2013history} for a description of the early
history of this problem).

To study primitive APs of squares, we study the multiple Dirichlet series
\begin{equation}\label{eq:Dsw_definition}
  \mathcal{D}(s, w)
  :=
  \sum_{\substack{m, h \geq 1 \\ (m,h)=1}}
  \frac{r_1(h)r_1(m)r_1(2m - h)}{m^s h^w},
\end{equation}
where $r_\ell(n)$ denotes the number of ways to represent $n$ as a sum of
$\ell$ squares. Thus $r_1(\cdot)$ is effectively a square
indicator function, and the numerator of this Dirichlet series
identifies whether $\{h, m, 2m - h\}$ is an AP of squares.

Our principal result is Theorem~\ref{thm:double_ds_decomposition}, which
states that $\mathcal{D}(s, w)$ has meromorphic continuation to $\mathbb{C}^2$ by means
of spectral expansion. We then exploit this meromorphic continuation to obtain
a variety of asymptotic results for the distribution of primitive APs.

Shifted convolutions of pairs of coefficients of modular forms frequently appear in analytic
number theory, and there exist several methods capable of handling them. In
contrast, triple shifted convolutions are typically poorly understood and have
fewer general techniques for analysis. Most other analyses follow the ideas and
methods of Blomer from~\cite{blomer2017triple}, which studies triple
convolutions involving divisor functions. As Blomer notes, it is possible to use
the circle method to study triple convolutions of coefficients of holomorphic
\emph{cusp forms}, but extending these techniques to other non-cuspidal modular
forms seems difficult.

In~\cite{HKLDWtriplecusp}, the authors produce a meromorphic continuation for a
triple shifted convolution Dirichlet series formed from holomorphic cusp forms
using spectral techniques. In this paper, we extend this analysis to
classical theta functions in order to study $\mathcal{D}(s, w)$. It is possible to extend
the techniques in this paper to other non-cuspidal holomorphic forms; the primary
difficulty in generalization lies in understanding growth of terms in the spectral
decomposition.

\section*{Outline of Paper and Results}

The paper ~\cite{HKLDWtriplecusp} concerns triple shifted convolutions of the form
\begin{equation} \label{eq:cusp_form_series}
  \sum_{m, h \geq 1}
  \frac{a(h) b(m) c(2m - h)}{m^s h^w},
\end{equation}
where the coefficients $a(\cdot)$, $b(\cdot)$, and $c(\cdot)$ are coefficients of
holomorphic cusp forms of full integral weight.  At first glance, the challenge in adapting methods from the cuspform case~\eqref{eq:cusp_form_series} to $\mathcal{D}(s,w)$ appears principally technical. However, closer inspection reveals that the spectral behavior of $\mathcal{D}(s,w)$ is distinguished in a way that requires significantly more specificity and care.

We begin in Section~\ref{sec:connections} with an overview of some classical counting
problems which can be studied using $\mathcal{D}(s,w)$. In particular, we discuss
the asymptotics of primitive APs of squares, as noted above, as well as their
connections to Pythagorean triples, congruent numbers, and rational points on
circles. In particular, since $\{a^2, b^2, c^2 \}$
corresponds to a rational point $(a/b, c/b)$ on the circle $x^2 + y^2 = 2$,
counts for APs of squares relate to counts for rational points on the circle of
radius $\sqrt 2$.

In Section~\ref{sec:singleds}, we study the single Dirichlet series
\begin{equation}
  D_h(s) = \sum_{m \geq 1} \frac{r_1(m)r_1(2m - h)}{m^s}
\end{equation}
for a fixed $h$. We obtain this series as an integral involving
$\Im(z)^{1/2}\theta(2z)\overline{\theta(z)}$,
which must be regularized at the cusps of $\Gamma_0(8)$ for the sake of
convergence. We produce a spectral expansion for this regularized form in
Section~\ref{sec:spectral_expansion}. Once simplified, this expansion involves only a simple
term introduced in regularization and a sum over dihedral Maass forms on
$\Gamma_0(8)$. In Section~\ref{sec:doubleds}, we study the full double
Dirichlet series $\mathcal{D}(s, w)$ and, in Theorem~\ref{thm:double_ds_decomposition},
we deduce its meromorphic continuation from its spectral expansion.

In Sections~\ref{sec:applications}
and~\ref{sec:further_applications},
we apply the meromorphic continuation of $\mathcal{D}(s, w)$ to a variety of
problems on APs of squares.
 For example,  in Section~\ref{ssec:ratl_points}, we prove the following theorem.
\begin{theorem*}[Theorem~\ref{thm:ratl_points}]
  Fix $\delta \in [0,1]$. For any $\epsilon > 0$, the number of primitive APs of squares $\{a^2, b^2, c^2\}$ with $b^{2} \leq X$ and $(a/b)^{2}
  \leq \delta$ is
  \begin{equation}
    \frac{2}{\pi^2} \arcsin(\sqrt{\delta / 2}) X^{\frac{1}{2}}
    + O_\epsilon(X^{\frac{3}{8}+\epsilon}).
  \end{equation}
\end{theorem*}
As shown in Section~\ref{sec:connections}, the main term above agrees with known results
concerning the equidistribution of rational points on the circle.

In
Section~\ref{sec:further_applications}, we prove three more theorems as further
applications. First, we count primitive APs of squares with bounded maximum.

\begin{theorem*}[Theorem~\ref{thm:boundedmax}] For any $\epsilon > 0$, the number of
primitive APs of squares $\{a^2,b^2,c^2\}$ with $c^2 \leq X$ is
\[\frac{\sqrt{2}}{\pi^2} \log(1+\sqrt{2})X^{\frac{1}{2}} + O_\epsilon\big(X^{\frac{3}{8}+\epsilon}\big).\]
\end{theorem*}

This result is then applied to count primitive APs of squares with
independently bounded first
and second terms.

\begin{theorem*}[Theorem~\ref{thm:naturalXY}]
Suppose that $Y \leq X$. For any $\epsilon > 0$, the number of primitive APs of
squares $\{a^2,b^2,c^2\}$ for which $a^2 \leq Y$ and $b^2 \leq X$ is
\begin{equation}
\begin{split}
  \frac{1}{\sqrt{2} \pi^2} Y^{\frac{1}{2}} \log\big(X/Y\big)
  + c\,Y^{\frac{1}{2}}
  +
  O_\epsilon\big(X^\epsilon Y^{\frac{3}{8}+\epsilon}\big),
\end{split}
\end{equation}
in which $c = \sqrt{2}\big(1+\frac{3}{2}\log 2 -\log(1+\sqrt{2})\big)/\pi^2$.
\end{theorem*}

Lastly, we count primitive APs of squares in which the product of the first two terms is
bounded.

\begin{theorem*}[Theorem~\ref{thm:firsttwoX}]
For any $\epsilon > 0$, the number of primitive APs of squares $\{a^2,b^2,c^2\}$ for which $ab \leq X$ is
  \begin{equation}
  \begin{split}
    \frac{2\sqrt{2}}{\pi^2} \, _2F_1(\tfrac{1}{4},\tfrac{1}{2},\tfrac{5}{4},\tfrac{1}{2})
	X^{\frac{1}{2}}
    +
    O_\epsilon\big(X^{\frac{3}{8}+\epsilon} \big).
  \end{split}
  \end{equation}
\end{theorem*}


\section*{Acknowledgements} CK is supported in part by
NSFC (No.11901585). DLD gratefully acknowledges support from
EPSRC Programme Grant EP/K034383/1 LMF:\ L-Functions and Modular Forms and
support from the Simons Collaboration in Arithmetic Geometry, Number Theory,
and Computation via the Simons Foundation grant 546235. The authors would like
to thank Dan Bump, Sol Friedberg, Jeff Hoffstein, Henryk Iwaniec, Alex
Kontorovich, Min Lee, Philippe Michel, and Paul Nelson for many helpful
conversations.

\section{Connections to Rational Points and Right Triangles}%
\label{sec:connections}

Before proving our main results, we consider connections to rational points
on the circle $x^2 + y^2 = 2$ and to integer right triangles.

\subsection{Equidistribution of points on the circle}%
\label{ssec:equidistribution}



In Theorem~\ref{thm:ratl_points}, we consider the number of primitive
APs of squares $\{a^2,b^2,c^2\}$ for which $(a/b)^2\leq \delta$. This result
can also be seen through the lens of equidistribution.

To see this connection, note that $a^2 + c^2 =
2b^2$ in an AP of squares, and hence $(a/b, c/b)$ is a rational point on the
circle $x^2 + y^2 = 2$. Let $A(b)$ denote the number of rational points on
$x^2 + y^2 = 2$ of the (reduced) form $(a/b, c/b)$. We see that
\begin{equation}
  \sum_{d \mid b} A(d)
  = \#\{(a, c) \in \mathbb{Z}^2 : a^2 + c^2 = 2b^2\}
  = r_2(2b^2) = r_2(b^2).
\end{equation}
Recalling that $r_2(n)/4$ is multiplicative, we can compute the Dirichlet series
\begin{equation} \label{eq:D_0(s)_variant}
  \sum_{n \geq 1} \frac{A(n)}{n^s}
  =
  \frac{4 \zeta(s) L(s, \chi_4)}{(1 + 2^{-s})\zeta(2s)},
\end{equation}
where $\chi_4 = (\frac{-1}{\cdot})$ is the non-trivial character of modulus $4$.
An application of Perron's formula and trivial estimates show that the number of
rational points on $x^2 + y^2 = 2$ of the (reduced) form $(a/b, c/b)$ with $b
\leq \sqrt{X}$ is
\begin{equation} \label{eq:primitive_point_count}
  \sum_{b \leq \sqrt{X}} A(b)
  =
  \frac{4}{\pi} X^\frac{1}{2} + O(X^{\frac{1}{3} + \epsilon}).
\end{equation}
%
%

If we assume that rational points on the circle equidistribute with respect to
arc length as their denominators grow, then we should expect the number of rational
points $(a/b, c/b)$ on $x^2 + y^2 = 2$ in the first quadrant with $b \leq
\sqrt{X}$ and $(a/b) \leq \sqrt{\delta}$ to be approximately
\begin{equation}
  \frac{\arcsin(\sqrt{\delta / 2})}{2\pi} \cdot \frac{4}{\pi} X^\frac{1}{2}
  = \frac{2}{\pi^2} \arcsin(\sqrt{\delta/2}) X^\frac{1}{2}.
\end{equation}
This agrees exactly with the main term
which appears in Theorem~\ref{thm:ratl_points}.

\begin{remark}
  An elementary proof of the equidistribution of rational points on $x^2+y^2=1$
  with respect to arc length as the size of the denominators grow can be found,
  for example, in~\cite{TB18}.  The methods applied therein can be adapted to
  the circle $x^2+y^2 =2$ via the linear map $(x,y) \to (x+y,x-y)$.
  %
  More generally,
  there is a deep and rich literature on studying aspects of equidistribution
  on spheres, and more generally on varieties. The analogous case in $3$
  dimensions on the unit sphere is proven in~\cite{duke2003rational}.
However, the authors are not aware of any equidistribution results for rational points
 on varieties which employ the properties of multiple Dirichlet series.
\end{remark}

\subsection{Right triangles}

There is a well-known one-to-one correspondence between APs of squares with
common difference $t$ and right triangles with area $t$, given by
\begin{align}
  \big\{
    (a, b, c): b^2 - a^2 = t = c^2 - b^2
  \big\}
  \leftrightarrow&
  \big\{ (\alpha,\beta,\gamma): \alpha^2 + \beta^2 = \gamma^2, \; \alpha\beta/2 = t \big\}
  \\
  (a, b, c) \mapsto (c-a, c + a, 2 b),
  \quad&
  (\alpha,\beta,\gamma) \mapsto \big(
    \tfrac{\beta-\alpha}{2}, \tfrac{\gamma}{2}, \tfrac{\beta+\alpha}{2}
  \big). \label{eq:triangle_bijection}
\end{align}
Thus counts of primitive APs of squares can lead to counts for primitive
Pythagorean triples.

In particular, each of our main theorems implies a corresponding result about
the number of primitive right triangles under certain constraints.  For example,
Theorem~\ref{thm:ratl_points} implies that the number of primitive right
triangles with hypotenuse at most $X$ and whose acute angles lie within
$\omega$ of $\frac{\pi}{4}$ is
\[\frac{2\omega}{\pi^2} X + O_\epsilon\big(X^{\frac{3}{4}+\epsilon}\big),\]
for any $\epsilon > 0$.  Through the same correspondence,
Theorem~\ref{thm:boundedmax} provides a count for primitive right triangles
$(\alpha,\beta,\gamma)$ for which $\alpha+\beta$ is bounded and
Theorem~\ref{thm:naturalXY} yields a count for primitive triangles of bounded
hypotenuse and separately bounded difference in leg length.


\section{The Single Dirichlet Series}\label{sec:singleds}

Let $\theta(z) = \sum_{n \in \mathbb{Z}} e(n^2 z) = \sum_{n \geq 0} r_1(n)
e(nz)$ denote the classical theta function, where $e(z) = e^{2 \pi i z}$. Then
$\theta(z)$ is a modular form of weight $1/2$ on $\Gamma_0(4)$. Consider also
the weight $0$, level $8$ Eisenstein series and Poincar\'{e} series with
character $\chi(d) = \big( \frac{2}{d} \big)$, given by
\begin{align}
  E(z, s; \chi)\label{eq:eisenstein_definition}
  &=
  {\mkern-18mu}
  \sum_{\gamma \in \Gamma_\infty \backslash \Gamma_0(8)}
  {\mkern-20mu}
  \chi(\gamma) \Im(\gamma z)^s
  =
  y^s
  + \sum_{c > 0} \sum_{\substack{d \in \mathbb{Z} \\ (8c, d) = 1}}
  \frac{y^s \big( \frac{2}{d} \big)}{\lvert 8cz + d \rvert^{2s}},
  \\
  P_h(z, s; \chi)\label{eq:poincare_definition}
  &=
  {\mkern-20mu}
  \sum_{\gamma \in \Gamma_\infty \backslash \Gamma_0(8)}
  {\mkern-20mu}
  \chi(\gamma) \Im(\gamma z)^s e(h \gamma z).
\end{align}
These sums converge absolutely for $\Re s >1$ and extend meromorphically to $s
\in \mathbb{C}$. Here and henceforth, we use $x$ and $y$ to represent the real
and imaginary parts of the complex number $z$.

In this section, we construct the single Dirichlet series
\begin{equation}\label{eq:single_ds_def}
  D_h(s) :=  \sum_{m \geq 1} \frac{r_1(m) r_1(2m - h)}{m^s}
\end{equation}
by studying the Petersson inner product
\begin{equation}
  \langle
  y^{\frac{1}{2}} \theta(2z) \overline{\theta(z)}
  - E(z, \tfrac{1}{2}, \chi), P_h(z, \overline{s}; \chi).
  \rangle,
\end{equation}
We establish in this section that $\sqrt{y} \theta(2z)\overline{\theta(z)} - E(z, \tfrac{1}{2},
\chi) \in L^2(\Gamma_0(8) \backslash \mathcal{H}; \chi)$. This result is used in
Section~\ref{sec:spectral_expansion} to study the spectral expansion and
meromorphic continuation of~\eqref{eq:single_ds_def}.

\subsection{Background on $\theta(z)$}\label{ssec:theta}

Recall that $\theta(z)$ is a modular form of weight $\frac{1}{2}$ on
$\Gamma_0(4)$. Under the action of $\gamma \in \Gamma_0(4)$, it transforms as follows:
\begin{equation}\label{eq:theta_multiplier}
  \theta(\gamma z)
  =
  \Big( \frac{c}{d} \Big) \varepsilon_d^{-1} (cz + d)^{\frac{1}{2}} \theta(z),
  \qquad \gamma = \begin{pmatrix} a&b \\ c&d \end{pmatrix} \in \Gamma_0(4),
\end{equation}
where $\varepsilon_d$ is $1$ if $d \equiv 1 \bmod 4$ and is $i$ if $d \equiv 3
\bmod 4$. The character $\big( \frac{c}{d} \big)$ refers to Shimura's extension
of the Jacobi symbol, and the square root $\sqrt z$ denotes the branch
$\sqrt z = \exp(\frac{1}{2} \log z)$ with the principal branch of the log.
The multiplier $\theta(\gamma z)/\theta(z)$ is the standard half-integral weight
multiplier. We also recall the identity
\begin{equation}\label{eq:theta_1_4}
  \theta(-1/4z) = \sqrt{- 2 i z} \theta(z).
\end{equation}

For notational convenience we define
\begin{equation}
	\Psi(z) := \theta(2z).
\end{equation}
From the computations
\begin{equation}
  2 \begin{pmatrix}
    a&b \\ c&d
  \end{pmatrix}
  z
  =
  \begin{pmatrix}
    a & 2b \\ \frac{c}{2} & d
  \end{pmatrix}
  2z
  \qquad \text{and} \qquad
  \Big( \frac{c/2}{d} \Big)
  = \Big( \frac{2}{d} \Big) \Big( \frac{c}{d} \Big),
\end{equation}
we see that  $\Psi$ transforms like
\begin{equation}\label{eq:psi_transformation_law}
  \Psi(\gamma z)
  =
  \Big( \frac{2}{d} \Big)
  \Big( \frac{c}{d} \Big) \varepsilon_d^{-1} (cz + d)^{\frac{1}{2}} \Psi(z),
  \qquad \gamma = \begin{pmatrix} a&b \\ c&d \end{pmatrix} \in \Gamma_0(8).
\end{equation}
Thus $\Psi(z)$ is a modular form of weight $\frac{1}{2}$, level $8$, and character
$\chi(d)=\big( \frac{2}{d} \big)$.

We also define
\begin{equation}
  V_1(z) := y^{\frac{1}{2}} \Psi(z) \overline{\theta(z)}.
\end{equation}
One can check that $V_1$ is an automorphic form of weight $0$,
level $8$, and character $\chi(d)$.

The congruence subgroup $\Gamma_0(8)$ has four cusps: $\infty$, $0$, $\frac{1}{2}$,
and $\frac{1}{4}$. We need to understand the behavior of $\theta$ and $\Psi$ at
each cusp. To describe this behavior, we follow Shimura~\cite{Shimura} and
Koblitz~\cite{Koblitz}.

Let $\GL(2, \mathbb{R})^+$ be the set of $2 \times 2$ matrices with positive
determinant. We define the metaplectic cover $\widetilde{G}$ to be the set of
pairs $(\gamma, \varphi(z))$, where $\gamma = \big( \begin{smallmatrix} a&b \\ c&d
\end{smallmatrix} \big) \in \GL(2, \mathbb{R})^+$ and $\varphi(z)$ is a
holomorphic function on the upper half-plane $\mathcal{H}$ such that
\begin{equation}
  \varphi(z)^2 = t \frac{cz + d}{\det(\gamma)^{1/2}},
  \quad \text{for some } t \in \mathbb{C}^\times
  \; \text{such that } \lvert t \rvert = 1.
\end{equation}
Then $\widetilde{G}$ is a group with the group law $(\gamma_1,
\varphi_1)(\gamma_2, \varphi_2) = (\gamma_1 \gamma_2, \varphi_1(\gamma_2 z)
\varphi_2(z))$.

The cover $\widetilde{G}$ surjects onto $\GL(2, \mathbb{R})^+$ through the
homomorphism $(\gamma, \varphi) \mapsto \gamma$, and we write the image of
this map as
\begin{equation}
	(\gamma, \varphi)_* := \gamma.
\end{equation}
In the other direction, for $\gamma \in \Gamma_0(8)$, define
\begin{equation}
 j_2(\gamma, z) = \frac{\theta(2\gamma z)}{\theta(2z)},
\end{equation}
which is the transformation law for $\Psi$ as in~\eqref{eq:psi_transformation_law}.
We define a homomorphism
\begin{equation}
	\gamma \mapsto \gamma^*
	:= (\gamma, j_2(\gamma, z)),
\end{equation}
which is one-to-one from $\Gamma_0(8)$ to $\widetilde{G}$.

For $\sigma = (\gamma, \varphi) \in \widetilde{G}$, we recall the definition of
the weight $k$ ($k \in \frac{1}{2}\mathbb{Z}$) slash operator on a function $f:
\mathcal{H} \longrightarrow \mathbb{C}$, given by
\begin{equation}
  f \big|_{[\sigma]}(z) := \frac{f(\gamma z)}{{\varphi(z)}^{2k}}.
\end{equation}
In what follows, we write $f |_{[\gamma]} = f |_{[\gamma^*]}$ for
$\gamma \in \Gamma_0(8)$ to denote the weight $k$ slash operator.  The weight,
either $\frac{1}{2}$ or $0$, will be clear from context.

A half-integral weight modular form $f$ on $\Gamma_0(8)$ with transformation law
$j_2$ admits Fourier expansions at each cusp $\mathfrak{a} \in \{\infty, 0,
\frac{1}{2}, \frac{1}{4} \}$ given by $f \big|_{[\sigma_{\mathfrak{a}}]}$ for
distinguished elements $\sigma_{\mathfrak{a}} \in \widetilde{G}$. Each
$\sigma_{\mathfrak{a}} \in \widetilde{G}$ projects to a classical scaling matrix
for $\mathfrak{a}$ so that $(\sigma_\mathfrak{a})^*(\infty) = \mathfrak{a}$. In addition, for
some $t \in \mathbb{C}^\times$ with $\lvert t \rvert = 1$,
$\sigma_{\mathfrak{a}}$ satisfies
\begin{equation}
  \sigma_{\mathfrak{a}}^{-1} \eta_{\mathfrak{a}}^* \sigma_{\mathfrak{a}}
  =
  (T, t),
\end{equation}
where $T = \big( \begin{smallmatrix} 1&1 \\ 0&1 \end{smallmatrix} \big)$ and
where $\eta_{\mathfrak{a}}$ generates the stabilizer $\Gamma_\mathfrak{a}$ of
the cusp $\mathfrak{a}$ in $\Gamma_0(8)$.

\subsection{Behavior at the cusps}\label{ssec:cusps}

Elements $\sigma_{\mathfrak{a}} \in \widetilde{G}$ for each cusp of
$\Gamma_0(8)$ are given by
\begin{align}
  \sigma_\infty &= \Big( \begin{pmatrix} 1 & 0 \\ 0 & 1\end{pmatrix}, 1 \Big),
  &\sigma_{0}    &= \Big(
    \begin{pmatrix} 0 & -1 \\ 8 & 0\end{pmatrix},
    \sqrt{-2\sqrt{2} z i}
  \Big),
  \\
  \sigma_{\frac{1}{2}}  &= \Big(
    \begin{pmatrix} 2 & 0 \\ 4 & 1\end{pmatrix},
    \sqrt{2 \sqrt 2 z + \tfrac{1}{\sqrt 2}}
  \Big),
  &\sigma_{\frac{1}{4}}  &= \Big(
    \begin{pmatrix} 1 & 0 \\ 4 & 1\end{pmatrix},
    \sqrt{4z + 1}
  \Big).
\end{align}
We allow $\widetilde{G}$ to act on $\mathcal{H}$ through $G$, and we write the
action as
\begin{equation}
  \sigma_\mathfrak{a} z := (\sigma_{\mathfrak{a}})_* z.
\end{equation}
By studying these actions, we compute the behavior of $\theta(z)$, $\Psi(z) = \theta(2z)$, and $V_1(z)$
at each cusp.

The behavior as $z \to i\infty$ is directly evident from the Fourier expansion
of $\theta$, and we have that
\begin{equation}
\begin{split}
  \theta(\sigma_\infty z)      &= 1 + O(e^{- 2 \pi y}),
  \qquad \Psi(\sigma_\infty z)  = 1 + O(e^{- 4 \pi y}),
  \\
  &V_1(\sigma_\infty z) = y^{\frac{1}{2}} \big( 1 + O(e^{-2 \pi y})\big).
\end{split}
\end{equation}
At the $0$ cusp, we have that
\begin{equation}
\begin{split}\label{eq:theta_psi_slash_0}
  \theta\big|_{[\sigma_0]}
  &= {(- 2 \sqrt{2} z i)}^{-\frac{1}{2}} \theta(-1/8z)
  = {(2)}^{\frac{1}{4}} \, \theta(2z)
  = {(2)}^{\frac{1}{4}} \, \Psi(z),
  \\
  \Psi\big|_{[\sigma_0]}
  &= {(- 2 \sqrt{2} z i)}^{-\frac{1}{2}} \Psi(-1/8z)
  = {(- 2 \sqrt{2} z i)}^{-\frac{1}{2}} \theta(-1/4z)
  = {(2)}^{-\frac{1}{4}} \, \theta(z),
\end{split}
\end{equation}
where we have used~\eqref{eq:theta_1_4} in each (and frequently in the sequel).
It follows that
\begin{equation}
	V_1 |_{[\sigma_0]}
	= y^{1/2}(\Psi \overline{\theta})|_{[\sigma_0]}
	= y^{1/2} \overline{\Psi} \theta = \overline{V_1},
\end{equation}
and thus
\begin{equation}
  V_1(\sigma_0 z) = y^{\frac{1}{2}} \big(1 + O(e^{-2 \pi y}) \big).
\end{equation}

The $1/2$ and $1/4$ cusps are similar to each other. Since these two cusps
play a much smaller role in this paper, we will only describe the $1/2$ cusp in detail.
We have that
\begin{align}
  \theta\big|_{[\sigma_{1/2}]}
  &= 2^{\frac{1}{4}} {(4z + 1)}^{-\frac{1}{2}} \theta\Big( \frac{2z}{4z + 1} \Big)
  = 2^{\frac{1}{4}} {(4z + 1)}^{-\frac{1}{2}} \theta\Big( \frac{1}{2 + 1/2z} \Big)
  \\
  &= 2^{\frac{1}{4}} {\Big( \frac{i}{4z} \Big)}^{\frac{1}{2}}
     \theta\Big( -\frac{1}{2} - \frac{1}{8z} \Big)
  = 2^{\frac{1}{4}} {\Big( \frac{i}{4z} \Big)}^{\frac{1}{2}}
     \Big( 2\theta(\tfrac{-1}{2z}) - \theta(\tfrac{-1}{8z}) \Big)
  \\
  &= 2^{\frac{1}{4}} \big( \theta(\tfrac{z}{2}) - \theta(2z) \big).
\end{align}
To pass from each line to the next, we apply~\eqref{eq:theta_1_4}. The
equality on the second line follows from the general identity
 $ \theta(z - \tfrac{1}{2}) = 2\theta(4z) - \theta(z)$,
which can be seen by comparing Fourier expansions.

Similarly, we compute that
\begin{equation}
  \Psi\big|_{[\sigma_{1/2}]}
  = 2^{\frac{1}{4}} \big( \tfrac{1+i}{2} \theta(z) - i \theta(4z) \big),
\end{equation}
%
where we have used the general identity $\theta(z - \tfrac{1}{4}) =
(1+i)\theta(4z) - i \theta(z)$.

Combining these together, it follows that
\begin{equation}
  V_1(\sigma_{1/2} \, z) = O(\sqrt y e^{-\pi y}).
\end{equation}
We note that the exponential decay comes from $\theta |_{[\sigma_{1/2}]}$,
whose constant Fourier coefficient vanishes.

The behavior at the cusp $\frac{1}{4}$ is similar. As $\theta$ is a modular form
on $\Gamma_0(4)$, we have that $\theta |_{[\sigma_{1/4}]} = \theta$. Analogous
computations to those with $\theta |_{[\sigma_{1/2}]}$ show that
\begin{align}
  \Psi\big|_{[\sigma_{1/4}]}
  &= (4z + 1)^{-\frac{1}{2}} \Psi\Big( \frac{z}{4z + 1} \Big)
  = (4z + 1)^{-\frac{1}{2}} \theta\Big( \frac{2z}{4z + 1} \Big)
  \\
  &= \theta(\tfrac{z}{2}) - \theta(2z),
\end{align}
which implies that
\begin{equation}
  V_1(\sigma_{1/4} \, z) = O(\sqrt y e^{- \pi y}).
\end{equation}
In this case, the exponential decay comes from $\Psi |_{[\sigma_{1/4}]}$, whose
constant Fourier coefficient vanishes.

\subsection{Constructing $V$}\label{ssec:constructing_v}

Now that we have established that $V_1$ decays exponentially at the $1/2$ and
$1/4$ cusps, and that $V_1$ and $V_1(\sigma_0z)$ grow like $\sqrt y + O(e^{-\pi
y})$, we construct a function $V$ from $V_1$ that decays exponentially at every
cusp.

Let $E(z, s; \chi)$ denote the Eisenstein
series~\eqref{eq:eisenstein_definition}. Define
\begin{equation}\label{eq:V_definition}
  V(z)
  := \sqrt y \theta(2z) \overline{\theta(z)} - E(z, \tfrac{1}{2}; \chi)
  = V_1(z) - E(z, \tfrac{1}{2}; \chi).
\end{equation}
The remainder of this section proves the following proposition.

\begin{proposition}
  The function $V(z)$ lies in $L^2(\Gamma_0(8) \backslash \mathcal{H}; \chi)$.
\end{proposition}

To prove this, we study the growth of $E(z, \tfrac{1}{2}; \chi)$ at each
cusp.  This information may be read from the constant terms in the Fourier
expansions of $E(z,\tfrac{1}{2};\chi)$ at each cusp.

For the cusp $\infty$, we elect to compute the full Fourier expansion
of $E(z,s;\chi)$ so as to avoid duplicate work in a later section.
This expansion is presented in the following lemma.

\begin{lemma} \label{lem:eisenstein_fourier_expansion}
The Fourier expansion of $E(z,s;\chi)$ is
\[E(z,s;\chi) = y^s + \sum_{h \neq 0} \rho_{y,s}(h) e^{2\pi i h x},\]
in which the coefficients $\rho_{y,s}(h)$ are defined by
\begin{equation} \label{eq:eisenstein_coefficient}
    \rho_{y, s}(h) = \frac{%
      \pi^s y^{\frac{1}{2}} \lvert h \rvert^{\frac{1}{2} - s}
      \sigma_{2s - 1}^\chi(h)
    }{%
      2^{6s - \frac{5}{2}} \Gamma(s) L(2s, \chi)
    }
    K_{s - \frac{1}{2}} (2 \pi \lvert h \rvert y).
  \end{equation}
\end{lemma}

\begin{proof}
Let $\delta_{ij}$ denote the Kronecker delta function.
Beginning from~\eqref{eq:eisenstein_definition}, we directly evaluate
\begin{align}
  &\int_0^1 L(2s, \chi) E(z, s; \chi) e^{- 2 \pi i h x} dx
  \\
  &\quad= L(2s,\chi)y^s \delta_{h0} + \sum_{c > 0} \frac{1}{(8c)^{2s}}
  \sum_{r = 0}^{8c-1} \sum_{m \in \mathbb{Z}} \int_0^1
  \frac{%
    y^s (\frac{2}{r}) e^{-2 \pi i h x}
  }{%
    \lvert z + m + \frac{r}{8c} \rvert^{2s}
  } dx
  \\
  &\quad= L(2s,\chi)y^s \delta_{h0} + \frac{1}{2^{6s}} \sum_{c > 0} \frac{1}{c^{2s}}
  \sum_{r = 0}^{8c-1} e\Big( \frac{hr}{8c} \Big) \Big( \frac{2}{r} \Big)
  \int_{-\infty}^\infty
  \frac{%
    y^s e^{- 2 \pi i h x}
  }{%
    (x^2 + y^2)^s
  } dx.
\end{align}
When $h=0$, the $r$-sum vanishes and only the $L(2s,\chi)y^s$ term survives. Otherwise,
we evaluate the integral as in~\cite[3.1.9]{Goldfeld06}. Writing $r = 8r' +
q$ with $0 \leq q < 8$ and $0 \leq r' < c$ transforms the sum over $r$ into the
product of an exponential sum and a Gauss sum. It ultimately follows that
\begin{equation}
  \sum_{r = 1}^{8c} e\Big( \frac{hr}{8c} \Big) \Big( \frac{2}{r} \Big)
  =
  \begin{cases}
    2\sqrt 2 c\, \big( \frac{2}{h/c} \big) & \text{if } c \mid h, \\
    0 & \text{otherwise}.
  \end{cases}
\end{equation}
Simplification and analytic continuation completes the proof.
\end{proof}

We note in particular that $E(\sigma_\infty z, s;\chi) = y^s + O(e^{-2\pi y})$, and
hence $V(z)$ vanishes at the cusp at $\infty$.

For the cusp at $0$, we compute that
\begin{equation}
  E(\sigma_0 \, z, s; \chi)
  =
  \sum_{\gamma \in \Gamma_\infty \backslash \Gamma_0(8)}
  \overline{\chi(\gamma)} \Im(\gamma(-1/8z))^s
  =
  \sum_{d > 0} \sum_{\substack{c \in \mathbb{Z} \\ (8c, d) = 1}}
  \frac{(y/8)^s (\frac{2}{d})}{\lvert dz - c \rvert^{2s}}.
\end{equation}
Thus the constant term in the Fourier expansion of $L(2s, \chi) E(\sigma_0\,z,
s; \chi)$ is
\begin{align}
  &\int_0^1 L(2s, \chi) E(\sigma_0 \, z, s; \chi) dx
  \\
  &\quad = \frac{1}{8^s} \sum_{d > 0} \frac{(\frac{2}{d})}{d^{2s}}
  \sum_{r = 1}^d \sum_{m \in \mathbb{Z}}
  \int_0^1 \frac{y^s}{\lvert z -m - \frac{r}{d}  \rvert^{2s}} \, dx
  \\
  &\quad = \frac{1}{8^s} \sum_{d > 0} \frac{(\frac{2}{d})}{d^{2s-1}}
  \int_{-\infty}^\infty \frac{y^s}{(x^2 + y^2)^s} dx
  =
  \frac{\sqrt{\pi} \Gamma(s - \frac{1}{2})}{8^s \Gamma(s)}
  L(2s - 1, \chi) y^{1 - s}.
\end{align}
It follows that the constant term in the Fourier expansion of $E(\sigma_0\, z,
\frac{1}{2}; \chi)$ is
\begin{equation}\label{eq:E0_constant_fourier_coeff}
  \lim_{s \to \frac{1}{2}}
  \frac{y^{1 - s} \sqrt{\pi} \Gamma(s - \frac{1}{2}) L(2s - 1, \chi)}
       {8^s \Gamma(s) L(2s, \chi)}
  = \sqrt{y},
\end{equation}
in which we have used the functional equation
\begin{equation}
  \Lambda(s, \chi)
  := (\pi/8)^{-\frac{s}{2}} \Gamma(\tfrac{s}{2}) L(s, \chi)
  = \Lambda(1 - s, \chi)
\end{equation}
to compute the limit. Thus $V_1$ cancels with $E(z, \tfrac{1}{2}; \chi)$ at the
$0$ cusp and $V$ vanishes there.

As $V_1$ vanishes at the $1/2$ and $1/4$ cusps, it remains only to show that
$E(z, \tfrac{1}{2}; \chi)$ vanishes there as well. For $1/4$, we have that
\begin{equation}
  E(\sigma_{1/4}\,z, s; \chi)
  =
  \sum_{\substack{c > 0 \\ 2 \nmid c}}
  \sum_{\substack{d \in \mathbb{Z} \\ (4c, d) = 1}}
  \frac{y^s (\frac{2}{d})}{\lvert 4cz + d \rvert^{2s}}.
\end{equation}
As $\chi(d - 4c) = -\chi(d)$ for odd $c$, we find $E(\sigma_{1/4}\,
(z+1), s; \chi) = - E(\sigma_{1/4}\, z, s; \chi)$. Thus the constant Fourier
term in $E(\sigma_{1/4}\, z, s; \chi)$ must vanish. Similarly,
\begin{equation}
  E(\sigma_{1/2}\, z, s; \chi)
  =
  \sum_{c > 0} \sum_{\substack{(4c, d) = 1 \\ d \equiv c \bmod 4}}
  \frac{(2y)^s  (\frac{2}{d})}{\lvert 4cz + d \rvert^{2s}},
\end{equation}
and one can check that $E(\sigma_{1/2}\,(z+1), s; \chi) = -E(\sigma_{1/2}\,z, s;
\chi)$. The constant Fourier term in $E(\sigma_{1/4}\, z, s; \chi)$
vanishes.

We conclude the $V(z)$ lies in $L^2(\Gamma_0(8) \backslash \mathcal{H}; \chi)$
as claimed.

\subsection{Constructing the single Dirichlet series}

We now construct and study $D_h(s)$ from~\eqref{eq:single_ds_def}. To do so, we examine the inner
product
\begin{equation}
  \langle V(z), P_h(z, \overline{s}; \chi) \rangle,
\end{equation}
where $V$ is as in~\eqref{eq:V_definition} and $P_h$ is the Poincar\'e
series~\eqref{eq:poincare_definition}. We will see that this inner product
encodes the single Dirichlet series $D_h(s)$.

\begin{proposition}\label{prop:single_ds_basic}
  For $h \geq 1$ and $\Re s \gg 1$, we have that
  \begin{equation}
	D_h(s)
    =
    \frac{%
      (8 \pi)^{s}
      \langle V, P_h(\cdot, \overline{s}+\tfrac{1}{2}; \chi)\rangle
    }{\Gamma(s)}
    +
    \frac{%
      2^{s}\sqrt{\pi} \sigma_0^\chi(h) \Gamma(s)
    }{\log(1 + \sqrt 2) h^{s} \Gamma(s+\frac{1}{2})},
  \end{equation}
  where
  \begin{equation}
    \sigma_{w}^\chi(h) = \sum_{d \mid h} \chi(d) d^w
  \end{equation}
  is a twisted divisor sum.
\end{proposition}

\begin{proof}

We evaluate $\langle V, P_h \rangle$ explicitly. As $P_h \in L^2(\Gamma_0(8)
\backslash \mathcal{H}; \chi)$ for sufficiently large $\Re(s)$, we can consider the inner products against the
two parts of $V = V_1 - E$ separately. The inner product against $V_1 = \Im(\cdot)^{1/2} \Psi
\overline{\theta}$ yields the Dirichlet series, as is seen through the classical
unfolding argument:
\begin{align}
  &\langle
    \Im(z)^{\frac{1}{2}} \theta(2z) \overline{\theta(z)},
    P_h(z, \overline{s}; \chi)
  \rangle
  =
  \int_0^\infty \int_0^1 y^{s - \frac{1}{2}}
  \theta(2z)\overline{\theta(z)e(hz)}\, \frac{dx \, dy}{y}
  \\
  &\quad=
  \sum_{m_1, m_2} r_1(m_1) r_1(m_2) \int_0^\infty \int_0^1 y^{s - \frac{1}{2}}
  e(2m_1 z - m_2 \overline{z} - h\overline{z}) \, \frac{dx \, dy}{y}
  \\
  &\quad=
  \sum_{m=1}^\infty r_1(m)r_1(2m - h)
  \int_0^\infty y^{s - \frac{1}{2}} e^{-8 \pi m y} \frac{dy}{y}
  \\
  &\quad=
  \frac{\Gamma(s - \frac{1}{2})}{(8 \pi)^{s - \frac{1}{2}}}
  \sum_{m \geq 1} \frac{r_1(m) r_1(2m - h)}{m^{s - \frac{1}{2}}}.
\end{align}

The inner product of $P_h$ against the Eisenstein series essentially extracts the \mbox{$h$-th}
Fourier coefficient of the Eisenstein series, $\rho_{y, w}(h)$. A short
computation shows that
\begin{align}
  \langle E(z, w; \chi), P_h(z, \overline{s}; \chi ) \rangle
  =
  \int_0^\infty \rho_{y, w}(h) e^{- 2 \pi y h} y^{s - 1} \frac{dy}{y}.
\end{align}
A formula for the Fourier coefficient $\rho_{y,w}(h)$ appears in~\eqref{eq:eisenstein_coefficient}. Applying that identity and changing variables to simplify the integral, we rewrite the inner product as
\begin{equation}
  \frac{%
    \pi^w h^{\frac{1}{2} - w} \sigma_{2w - 1}^\chi(h)
  }{%
    2^{6w - \frac{5}{2}} \Gamma(w) L(2w, \chi)
  }
  \frac{1}{(2 \pi h)^{s - \frac{1}{2}}}
  \int_0^\infty
  K_{w - \frac{1}{2}}(y) e^{-y} y^{s - \frac{1}{2}} \frac{dy}{y}.
\end{equation}
The integral above appears in the integral table~\cite[6.621(3)]{GradshteynRyzhik07}. Applying the integral from the table and evaluating at $w = \frac{1}{2}$, we see that
\begin{equation}
  \langle E(z, \tfrac{1}{2}; \chi), P_h(z, \overline{s}; \chi ) \rangle
  =
  \frac{%
    \pi^{1 - s} \sigma_0^\chi(h) \Gamma(s - \frac{1}{2})^2
  }{%
    2^{2s - \frac{1}{2}} h^{s - \frac{1}{2}} L(1, \chi) \Gamma(s).
  }
\end{equation}
The class number formula gives $L(1, \chi) = \log(1 +
\sqrt{2})/\sqrt{2}$. After rearranging and shifting $s \mapsto s+\frac{1}{2}$, we complete the proof.
\end{proof}

\section{Spectral Expansion}\label{sec:spectral_expansion}

We now produce a spectral expansion for $D_h(s)$.
To do this, we provide a spectral expansion for $\langle V, P_h \rangle$
by spectrally expanding the Poincar\'e series $P_h$. This approach to
constructing and studying Dirichlet series is not new, and is now
well-understood. See for instance the appendix to~\cite{sarnak2001estimates},
the work of Hoffstein and Hulse~\cite{HoffsteinHulse13}, or previous work of the
authors~\cite{HulseGaussSphere}. But in contrast to these previous works, the
behavior of this spectral expansion is distinguished. We will see that the
continuous component of the spectrum vanishes and the discrete component
consists entirely of explicit dihedral Maass forms. Together, these allow for an
unusually descriptive understanding of the spectral behavior.

We remark that the relative thinness of the spectral support of $V$ is prefigured
by similar results of Nelson~\cite{Nelson21} on the spectral decomposition
of $y^{1/2}\vert \theta(z) \vert^2$. (See also~\cite[\S{1}]{Nelson19}.)
There, the discrete spectrum vanishes in entirety owing to the non-existence
of dihedral forms with trivial character.

As is summarized in \cite[\S 2.1.2.1]{michel2007}, forms in $L^2(\Gamma_0(8)
\backslash \mathcal{H}; \chi)$ can be decomposed as a spectral expansion over Maass forms and Eisenstein series. In particular, since $P_h(z,s;\chi) \in L^2(\Gamma_0(8)
\backslash \mathcal{H}; \chi)$ for sufficiently large $\Re(s)$, it has a spectral expansion of the form
\begin{align}
  P_h(z, s; \chi)
  = &\sum_j \langle P_h(\cdot, s; \chi), \mu_j \rangle \mu_j(z) \\
  &+ \sum_{\mathfrak{a}} \frac{1}{4\pi} \int_{-\infty}^\infty
  \langle
    P_h(\cdot, s; \chi), E_\mathfrak{a}(\cdot, \tfrac{1}{2} + it; \chi)
  \rangle
  E_\mathfrak{a}(z, \tfrac{1}{2} + it; \chi) dt.
\end{align}
Here, $\{\mu_j\}$ is an orthonormal basis of Hecke-Maass forms for
$L^2(\Gamma_0(8) \backslash \mathcal{H}; \chi)$, where $\mu_j$ has eigenvalue
$\frac{1}{4} + t_j^2$ and type $\frac{1}{2} + it_j$. The sum over
$\mathfrak{a}$ ranges over the cusps of $\Gamma_0(8)$ that are non-singular
with respect to $\chi$, which are the cusps at $0$ and $\infty$.

Inserting this into the inner product $\langle V, P_h \rangle$ yields the
spectral expansion
\begin{align}
  &\langle V(z), P_h(z, \overline{s}; \chi) \rangle
  = \sum_j \overline{\langle P_h(\cdot, \overline{s}; \chi), \mu_j \rangle}
  \langle V, \mu_j \rangle\label{eq:spectral_full}
  \\
  &\qquad \qquad + \sum_{\mathfrak{a}} \frac{1}{4\pi} \int_{-\infty}^\infty
  \overline{\langle
    P_h(\cdot, \overline{s}; \chi), E_\mathfrak{a}(\cdot, \tfrac{1}{2} + it; \chi)
  \rangle}
  \langle V, E_\mathfrak{a}(z, \tfrac{1}{2} + it; \chi)\rangle dt.
\end{align}
This expansion simplifies further, and to that end it is helpful that we take
a brief digression into the Maass forms of $L^2(\Gamma_0(8) \backslash \mathcal{H}, \chi)$.

\subsection{Maass forms and dihedral Maass forms.}
Each Maass form in this spectral expansion of \eqref{eq:spectral_full} has
a Fourier expansion of the form
\begin{equation}
  \mu_j(z)
  =
  \sqrt{y} \sum_{n \geq 0}
  \rho_j(n) K_{it_j}(2 \pi \lvert n \rvert y) e^{2 \pi i n x} \label{eq:maass},
\end{equation}
and associated simultaneous Hecke eigenvalues $\lambda_j(n)$.  These eigenvalues
satisfy the recurrence relation
\begin{equation}
  \lambda_j(p^{n+1})
  =
  \lambda_j(p)\lambda_j(p^n) - \chi(p) \lambda_j(p^{n-1})\label{eq:primepows}
\end{equation}
and are multiplicative, in that
$\lambda_j(mn)=\lambda_j(m)\lambda_j(n)$ for $(m,n)=1$.

We normalize each $\mu_j$ so that the basis $\{ \mu_j \}$ is orthonormal with respect to the Petersson inner product. Thus
for each Maass form there is a constant $\rho_j(1)$ such that $\rho_j(n) =
\rho_j(1) \lambda_j(n)$, and we may assume $\rho_j(1) \in \mathbb{R}$ without
loss of generality.

Several Maass forms in $L^2(\Gamma_0(8) \backslash \mathcal{H}, \chi)$ can be
described explicitly. These are Maass forms coming from Hecke characters
defined on ideals of $\mathbb{Q}(\sqrt 2)$ and are examples of \emph{dihedral}
Maass forms,  Maass forms whose $L$-functions are Hecke $L$-functions~\cite{LY02}.
Rather interestingly, these forms comprise the entirety of the spectral
expansion of $\langle V, P_h \rangle$.

For each $m \in \mathbb{Z}_{\neq 0}$, consider the function
\begin{equation} \label{eq:dihedral}
  f_m(z) = \!\sum_{n \geq 1} \sum_{N(\mathfrak{b})=n} \!\!\eta(\mathfrak{b})^m\!
  \sqrt{y} K_{\frac{i m \pi}{2 \log(1 + \sqrt{2})}}(2 \pi n y)
 (e(nx) +(-1)^m e(-nx))
\end{equation}
in which $\eta(\mathfrak{b})$ is the Hecke character defined on ideals of
$\mathbb{Q}(\sqrt 2)$ by
\begin{equation}
  \eta\big((a + b \sqrt 2)\big)
  =
  \sgn(a + b \sqrt 2) \sgn(a - b \sqrt 2)
  \Big\lvert
    \frac{a + b \sqrt{2}}{a - b \sqrt{2}}
  \Big\rvert^{\frac{i\pi}{2 \log(1 + \sqrt{2})}}.
\end{equation} We note that $f_m(z) = f_{-m}(z)$.
It suffices to define $\eta$ on principal ideals, as the ring of
integers in $\mathbb{Q}(\sqrt 2)$ is a principal ideal domain.
Following Maass, and as recounted in~\cite[Theorem 1.9.1]{Bump98}, the functions
$f_m(z)$ are indeed Maass cusp forms for $\Gamma_0(8)$ with nebentypus $\chi$ and
type $\frac{1}{2} + \frac{i m \pi}{2 \log(1 + \sqrt{2})}$. These forms have
multiplicative Hecke eigenvalues $\lambda_m(h)$ for which
\[
 f_m(z) = \!\sum_{n \neq 0} \lambda_m(n)
  \sqrt{y} K_{\frac{i m \pi}{2 \log(1 + \sqrt{2})}}(2 \pi |n| y)e(nx),
\] and these $\lambda_m(h)$ can be defined on
rational primes $p$ as
\begin{equation}\label{eq:dihedralcoeffs}
  \lambda_m(p) = \begin{cases}
  \eta^m(\mathfrak{p})+\eta^{-m}(\mathfrak{p})
   & \text{if } \chi(p) =1, \\
   0 & \text{if } \chi(p)=-1, \\
   (-1)^m & \text{if } p=2,
   \end{cases}
\end{equation}
where $\mathfrak{p}$ is a prime ideal in the integer ring of
$\mathbb{Q}(\sqrt{2})$ over $\mathbb{Z}$ such that $p$ splits as $(p) =
\mathfrak{p}\overline{\mathfrak{p}}$. We then define $\lambda_m$ for non-zero
integers via~\eqref{eq:primepows}, multiplicativity, and the relation
$\lambda_m(1) =(-1)^m\lambda_m(-1) = 1$.

In the case $m=0$, the function $f_m(z)$ does not define a cusp form.  Rather,
as noted in~\cite{Bump98}, $\frac{1}{2} \log(1+\sqrt{2})\sqrt{y} + f_0(z)$
defines an Eisenstein series on $\Gamma_0(8)$ with nebentypus $\chi$. A comparison
of Fourier coefficients confirms that this Eisenstein series is precisely
$\frac{1}{2} \log(1+\sqrt{2}) E(z,\frac{1}{2};\chi)$. In particular,
$\lambda_0(p) = \sigma_0^\chi(p)$ and the potential contribution of this missing
Maass form is exactly the contribution of the Eisenstein series in
Proposition~\ref{prop:single_ds_basic}.


The $L$-functions $L(s,f_m)$ attached to these Maass forms coincide with the Hecke
$L$-functions
\begin{equation}
L(s,\eta^{m}) = \sum_{\mathfrak{a}} \frac{\eta^{m}(\mathfrak{a})}{N(\mathfrak{a})^s}, \label{eq:heckeL}
\end{equation}
which are entire for $m \neq 0$. In the case $m=0$, $L(s,\eta^0) = \zeta(s)L(s,\chi)$ is the Dedekind zeta function for $\mathbb{Q}(\sqrt{2})$ and admits a simple pole at $s=1$.
By examining the conductors of the general functional equation for Hecke $L$-functions, as in~\cite{Bo05}, we have that the only dihedral Maass cusp forms of level 8 and nebentypus $\chi$ are exactly our $f_m$ described above.

We may now state the main theorem of this section.

\begin{theorem}\label{thm:single_ds_decomposition}
 For $h \geq 1$ and $\Re s \gg 1$, we have that
    \begin{align}\label{eq:in_thm_single_ds_decomposition}
    D_h(s) & = \sum_{m \geq 1} \frac{r_1(m) r_1(2m-h)}{m^s}
    \\
    &=  \sum_{m \in \mathbb{Z}} \frac{(-1)^m\lambda_m(h)}{h^s}
    \frac{\Gamma(s +\frac{im\pi}{2\log(1+\sqrt{2})})\Gamma(s-
    \frac{im\pi}{2\log(1+\sqrt{2})})}{2^{1-3s}\log(1+\sqrt{2})\Gamma(2s)},
\end{align}
  where $\lambda_m(h)$ are defined as in~\eqref{eq:dihedralcoeffs}.
\end{theorem}

The proof of this decomposition is contained in the following subsections.  We first compute the continuous part of the spectrum, which vanishes, then demonstrate that only dihedral Maass forms contribute to the discrete part of the spectrum.  A final bit of simplification completes the proof.

\subsection{Continuous spectrum}

The continuous spectrum component of~\eqref{eq:spectral_full} comes from
Eisenstein series associated to the non-singular cusps, $0$ and $\infty$. We will show that
both terms vanish.

\begin{lemma}\label{lem:continuous_vanishes}
  We have that $\langle V, E_{\mathfrak{a}}(\cdot, \overline{s}; \chi) \rangle =
  0$ for the cusps $0$ and $\infty$.
\end{lemma}

\begin{proof}

For the cusp at $\infty$, we directly compute
\begin{equation}
  \langle V, E_\infty(z, \overline{s}; \chi) \rangle
  =
  \int_0^\infty \int_0^1 V(z) y^{s - 1} \frac{dx \, dy}{y}.
\end{equation}
The integral over $x$ extracts the constant term in the Fourier expansion of
$V(z)$. As $V(z) = \sqrt{y} \theta(2z)\overline{\theta(z)} - E(z, \frac{1}{2};
\chi)$, the constant Fourier coefficient is exactly
\begin{equation}
  \sqrt{y} \sum_{m \geq 0} r_1(m) r_1(2m) e^{-8\pi m y} - \sqrt{y}.
\end{equation}
Since $r_1(m)r_1(2m) = 0$ except when $m = 0$, in which case it is $1$, we see
that the constant Fourier coefficient is identically $0$. Thus $\langle V,
E_\infty \rangle = 0$.

For the cusp at $0$, we first note that $\sigma_0 = \big( \begin{smallmatrix}
0&-1 \\ 8&0 \end{smallmatrix} \big)$ is an involution on $\mathcal{H}$ and that
\begin{equation}
  \begin{pmatrix}
    0&-1 \\ 8&0
  \end{pmatrix}
  \begin{pmatrix}
    a&b \\ 8c&d
  \end{pmatrix}
  =
  \begin{pmatrix}
    d&-c \\ -8b&a
  \end{pmatrix}
  \begin{pmatrix}
    0&-1 \\ 8&0
  \end{pmatrix}.
\end{equation}
Thus $\sigma_0\big( \Gamma_0(8) \backslash \mathcal{H} \big) = \Gamma_0(8)
\backslash \mathcal{H}$ and it immediately follows that $E_0(z, s; \chi) =
E_{\infty}(\sigma_0z, s; \chi)$. Within the inner product
\begin{equation}
  \langle V, E_0(\cdot, \overline{s}; \chi) \rangle
  =
  \iint_{\Gamma_0(8) \backslash \mathcal{H}}
  {\mkern-15mu}
  V(\sigma_0(z)) \overline{E_{\infty}(z, \overline{s}; \chi)} \, d\mu
  =
  \int_0^\infty
  {\mkern-15mu}
  \int_0^1 V(\sigma_0 z) y^{s - 1} \frac{dx dy}{y},
\end{equation}
the integral over $x$ extracts the constant term in the Fourier expansion of
$V(\sigma_0 z)$. Earlier we noted that~\eqref{eq:theta_psi_slash_0} implies that
$\Psi\overline{\theta}|_{[\sigma_0]} = \overline{\Psi}\theta$; and
in~\eqref{eq:E0_constant_fourier_coeff}, we computed that the constant term of
$E_{\infty}(\sigma_0z, \frac{1}{2}; \chi)$ is $\sqrt y$. Thus $\langle V, E_0(z,
\overline{s}; \chi)\rangle = 0$ as well.
\end{proof}

\subsection{Discrete spectrum}

The integral formula~\cite[6.621(3)]{GradshteynRyzhik07} and unfolding produces the identity
\begin{equation}
  \overline{\langle P_h(\cdot, \overline{s}; \chi), \mu_j \rangle}
  =
  \frac{%
    \rho_j(h) \sqrt{\pi}
    \Gamma(s - \frac{1}{2} + it_j) \Gamma(s - \frac{1}{2} - it_j)
  }{(4\pi h)^{s - \frac{1}{2}} \Gamma(s)}
\end{equation}
for the inner product of the Poincar\'e series and a Maass form of type $\frac{1}{2}+it_j$.

By applying this expression and the result of Lemma~\ref{lem:continuous_vanishes} to the spectral expansion~\eqref{eq:spectral_full}, we obtain the following lemma.

\begin{lemma}\label{lem:discrete_basic}
  For $h \geq 1$ and $\Re s \gg 1$,
  \begin{equation}
    \langle V, P_h(\cdot, \overline{s}; \chi) \rangle
    =
    \sum_j \frac{\rho_j(h) \sqrt \pi}{(4 \pi h)^{s - \frac{1}{2}}}
    G(s, it_j) \langle V, \mu_j \rangle,
  \end{equation}
  in which
  $G(s, z) := \Gamma(s - \frac{1}{2} + z) \Gamma(s - \frac{1}{2} - z) /
  \Gamma(s) $.
\end{lemma}

We now investigate the individual terms in the discrete spectrum.  Just as the
inner products $\langle V, E_\mathfrak{a} \rangle$ caused the continuous
spectrum to vanish, the inner products $\langle V, \mu_j \rangle$ cause many
terms in the discrete spectrum to vanish.

\begin{lemma} \label{lem:self_dual_survives}
We have $\langle V, \mu_j \rangle \neq 0$ if and only if $\mu_j = \langle
f_m, f_m \rangle^{-1/2}  f_m$ for some $m \in \mathbb{N}$ as
in~\eqref{eq:dihedral}, in which case
\begin{equation}\label{eq:inprod}
  \rho_j(1)\langle V, \mu_j \rangle =\frac{2  (-1)^m}{\log(1+\sqrt{2})}.
\end{equation}
\end{lemma}

To prove this lemma, we recognize $\theta$ as the residue of a weight $1/2$
Eisenstein series. We use this Eisenstein series in place of $\theta$ in
the inner product $\langle V, \mu_j \rangle$ to interpret the inner product using a Rankin-Selberg
type convolution. This object has no pole unless $\mu_j$ is self-dual, which coincides exactly with $\mu_j$ being dihedral, and implies that the associated residue must be zero in the non-dihedral case.


\begin{proof}

We define the weight $1/2$, level $8$ Eisenstein series
\begin{equation}
  E^{\frac{1}{2}}(z, w; \Gamma_0(8))
  :=
  \sum_{\gamma \in \Gamma_\infty \backslash \Gamma_0(8)}
  \Im(\gamma z)^w J(\gamma, z)^{-1},
\end{equation}
in which
\begin{equation}
	J(\gamma,z)  := j(\gamma,z) / \lvert j(\gamma,z) \rvert
\end{equation}
is a normalization of the theta multiplier $j(\gamma,z) = \theta(\gamma z)/
\theta(z)$ given in~\eqref{eq:theta_multiplier}. We note that this matches the
definitions of the half-integral weight Eisenstein series defined
in~\cite{GoldfeldHoffstein85} and~\cite{shimura1985eisenstein}, except that we
normalize the metaplectic cocycle. For comparison, if $\mathcal{E}(z, w)$
denotes the Eisenstein series in either of these works, then $E^{\frac{1}{2}}(z,
w; \Gamma_0({8})) = \Im(z)^{1/4} \mathcal{E}(z, w - \frac{1}{4})$. This
particular normalization agrees with the normalization of the metaplectic
Eisenstein series appearing in~\cite[\S13]{Iwaniec97} and
in~\cite{LowryDudaThesis}.

It is known~\cite[Theorem 2.3]{shimura1985eisenstein} that $E^{\frac{1}{2}}(z,
w; \Gamma_0(8))$ has a simple pole at $w = \frac{3}{4}$ with residue of the form
$ y^{1/4} g(z)$, where $g(z)$ is a holomorphic form of weight $\frac{1}{2}$ and level 8.
Since the space of such forms is one-dimensional, we have that $g(z) =
c^{-1}\theta(z)$ where $c$ can be computed using the methods
of~\cite{GoldfeldHoffstein85} to be
\begin{equation}
c^{-1}= \frac{1}{ 4\pi}.  \label{eq:cvalue}
\end{equation}
Since $\theta(z)$ is (up to a constant) the residue of $E^{\frac{1}{2}}(z, w;
\Gamma_0(8))$ at $w = \frac{3}{4}$, the same holds for $\overline{\theta(z)}$ and $\overline{E^{\frac{1}{2}}(z, \overline{w};
\Gamma_0(8))}$. We apply the latter to study $\langle V, \mu_j\rangle$, and compute that
\begin{align}
  \langle V, \mu_j\rangle
  &=
  \langle y^{\frac{1}{2}} \overline{\theta(z)} \theta(2z), \mu_j \rangle
  =
  c\, \Res_{w = \frac{3}{4}}
  \langle
    y^{\frac{1}{4}}\theta(2z) \overline{E^{\frac{1}{2}}(z, \overline{w}; \Gamma_0(8))},
    \mu_j
  \rangle
  \\
  &=
  c\, \Res_{w = \frac{3}{4}} \int_0^\infty \int_0^1 y^{\frac{1}{4}}
\theta(2z) y^w \overline{\mu_j(z)} \frac{dx dy}{y^2}
  \\
  &= c \,\Res_{w = \frac{3}{4}}
  \sum_{n \geq 1} r_1(n) \overline{\rho_j(2n)} \int_0^\infty y^{w - \frac{1}{4}}
  K_{it_j} (2\pi n y) e^{- 2 \pi n y} \frac{dy}{y}
  \\
  &=\label{line:Vmu_discrete_potential_pole}
 c\, \Res_{w = \frac{3}{4}}
  \frac{%
   2^{\frac{7}{4}} \pi^{\frac{3}{4}} \Gamma(w - \frac{1}{4} + it_j)\Gamma(w - \frac{1}{4} - it_j)
  }{%
    (8\pi)^w \Gamma(w + \frac{1}{4})
  }
  \sum_{n \geq 1} \frac{\overline{\rho_j(2n^2)}}{n^{2w - \frac{1}{2}}}.
\end{align}
The second line comes from unfolding the Eisenstein series. The third line
follows after expanding $\theta$ and $\mu_j$ in Fourier expansions, performing
the integral over $x$ to extract the constant term of
$\overline{\theta \mu_j}$, and recognizing that $\overline{K_{it_j}(y)} =
K_{it_j}(y)$ when $t_j \in \mathbb{R} \cup i\mathbb{R}$ and $y \in \mathbb{R}^+$.
The fourth line follows after computing the integral (which is
computed in~\cite[6.621(3)]{GradshteynRyzhik07}).

The gamma functions are holomorphic at $w = 3/4$, hence the only
potential source of a pole in~\eqref{line:Vmu_discrete_potential_pole} is the
Dirichlet series. This series differs trivially from the symmetric square series
associated to $\mu_j$,
\begin{equation}
L(s,\Sym^2 \mu_j) = L(2s,\chi^2) \sum_{n \geq 1} \frac{\lambda_j(n^2)}{n^{s}},
\end{equation}
only in the 2-factor, hence $\langle V,\mu_j\rangle$ is non-zero if and only if
$L(s,\Sym^2 \mu_j)$ has a pole at $s=1$. The identity
\[L(s,\Sym^2 \mu_j)  = \frac{L(s,\mu_j \otimes \mu_j)}{L(s,\chi)}\]
implies that $L(s,\Sym^2 \mu_j)$ has a simple pole at $s=1$ if and only if
$\mu_j$ is self-dual; that is, if $\overline{\lambda_j(n)} = \lambda_j(n)$ for
all $n \in \mathbb{Z}$. Since $\chi$ is a nontrivial character and
$\overline{\lambda_j(n)} = \chi(n) \lambda_j(n)$, we have $\lambda_j(n) =0$
whenever $\chi(n) =-1$ and so $\mu_j$ is unchanged when twisted by $\chi$. As
is noted in~\cite{KS02}, if $\mu_j = \mu_j \otimes \chi$ for a nontrivial Maass
form, then $\mu_j$ is a multiple of a dihedral Maass form and, for level 8
Maass forms of weight zero and with nebentypus $\chi$, these must exactly be
forms of the kind given in~\eqref{eq:dihedral}, as discussed following~\eqref{eq:heckeL}.
Hence $\mu_j = \rho_m(1)f_m$
for some $m \in \mathbb{N}$ and $\rho_m(1) =\langle f_m,f_m\rangle^{-1/2}
= \rho_j(1)$.

To prove~\eqref{eq:inprod} and complete the proof of
Lemma~\ref{lem:self_dual_survives}, note from $\lambda_m(2n) =
(-1)^m\lambda_m(n)$ that the series in~\eqref{line:Vmu_discrete_potential_pole}
is
\begin{equation}
  \sum_{n \geq 1} \!\frac{\overline{\rho_j(2n^2)}}{n^{s}}
  =  \rho_m(1)(-1)^m \frac{L(s,f_m \otimes f_m)}{L(2s,\chi^2)L(s,\chi)} \label{eq:rankin1}
\end{equation}
Combined with~\eqref{eq:cvalue} and the identity $\cosh(\pi t_j)
\Gamma(\frac{1}{2} - it_j)\Gamma(\frac{1}{2} + it_j) = \pi$, we
rewrite~\eqref{line:Vmu_discrete_potential_pole} in the form
\begin{align} \label{eq:simplified_line:Vmu_discrete_potential_pole}
  \langle V, \mu_j\rangle
  &=\frac{\rho_m(1)(-1)^m\sqrt{2}\,\pi^2}{\cosh(\frac{m\pi^2}{2\log(1+\sqrt{2})})L(2,\chi^2)L(1,\chi)} \Res_{w=1} L(w,f_m \otimes f_m).
\end{align}
By unfolding the inner product of $\vert f_m \vert^2$ against the level $8$
Eisenstein series, we also produce
\begin{equation} \label{eq:rankinres2}
  \Res_{w = 1} L(w, f_m \otimes f_m)
  =
  \tfrac{1}{8}\cosh\Big(\tfrac{m\pi^2}{2 \log(1+\sqrt{2})}\Big)
\langle f_m, f_m \rangle.
\end{equation}
Substitution of~\eqref{eq:rankinres2}
into~\eqref{eq:simplified_line:Vmu_discrete_potential_pole} gives
\[
  \rho_j(1)\langle V,\mu_j \rangle =
  \frac{(-1)^m\sqrt{2}\,\pi^2}{8L(2,\chi^2)L(1,\chi)} = \frac{2 \cdot
  (-1)^m}{\log(1+\sqrt{2})},
\]
in which $\rho_m(1)^2 = \langle f_m,f_m \rangle^{-1}$, $L(2,\chi^2) = \frac{\pi^2}{8}$, and $L(1,\chi) = \log(1+\sqrt{2})/\sqrt{2}$ have been used in the simplification.
\end{proof}


\subsection{Proof of Theorem~\ref{thm:single_ds_decomposition}}

Theorem~\ref{thm:single_ds_decomposition} quickly follows from the previous
discussion. One substitutes the spectral expansion~\eqref{eq:spectral_full} into
the expression for $D_h(s)$ given in
Proposition~\ref{prop:single_ds_basic}. Lemma~\ref{lem:continuous_vanishes}
shows that the continuous spectrum vanishes, while
Lemmas~\ref{lem:discrete_basic} and~\ref{lem:self_dual_survives} give the form
of the discrete component. It follows that
   \begin{align}\label{eq:incomplete1}
    D_h(s)  &=   \frac{2^{1+s}\sqrt{\pi}}{\log(1+\sqrt{2})\Gamma(s)}\sum_{m \geq 1} \frac{(-1)^m\lambda_m(h)}{h^{s}}
    G\bigg(s+\tfrac{1}{2}, \frac{im\pi}{2\log(1+\sqrt{2})}\bigg)  \\
   &\qquad  +    \frac{%
      2^{s}\sqrt{\pi} \sigma_0^\chi(h) \Gamma(s)
    }{\log(1 + \sqrt 2) h^{s} \Gamma(s+\frac{1}{2})},
\end{align}
The theorem now follows from the gamma duplication formula and the identities $\sigma_0^\chi(h) = \lambda_0(h)$ and $\lambda_m(h) = \lambda_{-m}(h)$. \qed


\section{Constructing the Double Dirichlet Series}\label{sec:doubleds}

We now build upon the analysis of $D_h(s)$, the Dirichlet series introduced
in~\eqref{eq:single_ds_def}, to study the meromorphic properties of the double
Dirichlet series
\begin{equation}
\mathcal{D}(s, w) := \sum_{\substack{m, h \geq 1 \\ (m,h)=1}} \frac{r_1(h)r_1(m)r_1(2m-h)}{m^s h^w}.
\end{equation}
We prove the following theorem.
%
\begin{theorem}\label{thm:double_ds_decomposition}
  The double Dirichlet series $\mathcal{D}(s,w)$ has meromorphic continuation to
  $\mathbb{C}^2$. For $\Re s$ and $\Re w$ sufficiently large, we have
  \begin{align}
    \mathcal{D}(s,w) &= \frac{2^{3s}(1-2^{-2s-2w})}{\zeta^{(2)}(4s+4w) \log(1+\sqrt{2})\Gamma(2s)} \\
    &\quad \times\sum_{m \in \mathbb{Z}} {(-1)}^m L(2s+2w,\eta^{2m})
    \Gamma(s +it_m)\Gamma(s-it_m),\label{eq:in_thm_double_ds_decomposition}
  \end{align}
  in which $t_m = \frac{m\pi}{2\log(1+\sqrt{2})}$, $\zeta^{(2)}(s)
  =(1-\frac{1}{2^s})\zeta(s)$, and $L(s,\eta^m)$ is as in~\eqref{eq:heckeL}.
\end{theorem}

\begin{proof}
Theorem~\ref{thm:single_ds_decomposition} expresses $D_h(s)$ as a sum,~\eqref{eq:in_thm_single_ds_decomposition}, over $m \in \mathbb{Z}$.
By multiplying each term in~\eqref{eq:in_thm_single_ds_decomposition} by
$r_1(h)/h^w$ and summing over $h \geq 1$, we produce a decomposition for a
variant of $\mathcal{D}(s,w)$ which omits the coprime assumption.

The sum over $h \geq 1$ in that series may be written in the form
\begin{equation}
\sum_{h\geq 1} \frac{\lambda_m(h^2)}{h^{s}}
	= \frac{L(s,\chi^2)L(s,\eta^{2m})}{L(2s,\chi^2)}
	= \frac{\zeta^{(2)}(s)L(s,\eta^{2m})}{\zeta^{(2)}(2s)}
\end{equation}
by combining~\eqref{eq:rankin1} with the Euler product factorization $L(s,f_m \otimes f_m) = L(s,\chi^2)L(s,\chi)L(s,\eta^{2m})$.
Once simplified, we see that
\begin{align}
    &\sum_{m,h\geq 1} \frac{r_1(h)r_1(m)r_1(2m-h)}{m^sh^w} \\
    &\quad = \sum_{m \in \mathbb{Z}} \frac{{(-1)}^m L(2s+2w,\eta^{2m}) 2^{3s}\zeta^{(2)}(2s+2w)
    \Gamma(s +it_m)\Gamma(s-it_m)}{\zeta^{(2)}(4s+4w) \log(1+\sqrt{2})\Gamma(2s)}.
  \end{align}
The identity~\eqref{eq:in_thm_double_ds_decomposition} follows by dividing both
sides by $\zeta(2s+2w)$.

These $L$-functions have meromorphic continuation to all $s,w\in \mathbb{C}$ and grow
at most polynomially in $t_m$ in vertical strips. Thus Stirling's approximation
for the gamma functions gives normal convergence over $m$ and completes the
proof of the theorem.
\end{proof}

\begin{remark}
When $w=0$, the expansion~\eqref{eq:in_thm_double_ds_decomposition}
closely resembles the hyperbolic Fourier expansion of the level $1$, weight $0$ Eisenstein
series as described in~\cite[ch. 2, $\S$3]{Siegel80} or~\cite[\S{3.2}]{Goldfeld06}.  These
expansions differ from~\eqref{eq:in_thm_double_ds_decomposition} by also including the Hecke
$L$-functions of odd powers of $\eta$.
\end{remark}

\begin{remark}
The term $m=0$ is distinguished within the sum~\eqref{eq:in_thm_double_ds_decomposition} because $L(2s+2w, \eta^{2m})$ has a polar line at $2s+2w=1$ if and only if $m=0$. In Sections~\ref{sec:applications} and~\ref{sec:further_applications}, we show that this polar line is the source of the main terms in the asymptotic formulas presented in Theorems~\ref{thm:ratl_points},~\ref{thm:boundedmax},~\ref{thm:naturalXY}, and~\ref{thm:firsttwoX}.
\end{remark}

\begin{remark} The function $\zeta(2s+2w)\mathcal{D}(s,w)$ can be used to produce
counts for APs of squares which include both primitive and imprimitive APs. We do
not pursue this here, since counts for APs with unrestricted GCDs can be obtained
from the primitive case by purely elementary methods.
\end{remark}

\section{Weight functions}\label{sec:weight_functions}

To derive arithmetic results from Theorem~\ref{thm:double_ds_decomposition}, we
study double sums whose summands are $r_1(h)r_1(m)r_1(2m-h)$,
but whose range of summation is constrained. To
study these sums, we relate them to smooth analogues. In this section,
we describe the necessary smoothing weight functions and their properties.

We use two weight functions, $u_{+x}(t)$ and $u_{-x}(t)$. Define $u_{+x}(t)$ and
$u_{-x}(t)$ to be smooth, non-increasing functions with compact support,
satisfying
\begin{equation}
  u_{-x}(t) = \begin{cases}
    1 & t \leq 1 - \frac{1}{x}, \\
    0 & t \geq 1,
  \end{cases}
  \quad \text{and} \quad
  u_{+x}(t) = \begin{cases}
    1 & t \leq 1, \\
    0 & t \geq 1 + \frac{1}{x},
  \end{cases}
\end{equation}
where $x > 1$ is an optimizing parameter chosen in each application.
Let $U_{-x}(s)$ and $U_{+x}(s)$ denote the Mellin transforms of $u_{-x}(t)$ and
$u_{+x}(t)$, respectively. Trivial bounds, differentiation under the integral,
and the convexity principle (coupled with repeated integration by parts) show that
\begin{enumerate}
  \item $U_{\pm x}(s) = s^{-1} + O_s(1/x)$.
  \item $U_{\pm x}'(s) = s^{-2} + O_s(1/x)$.
  \item for all $\alpha \geq 1$, and for $s$ constrained in a vertical strip with
  $\lvert s \rvert > \epsilon$, we have
  \begin{equation}\label{eq:weight_Us_bound}
    U_{\pm x}(s)
    \ll_\epsilon
    \frac{1}{x} \Big( \frac{x}{1 + \lvert s \rvert} \Big)^{\alpha}.
  \end{equation}
\end{enumerate}

In each application, we construct two smoothed approximations
$S_{-x}$ and $S_{+x}$ to a desired sum $S$, formed from smoothing $S$ with
$u_{-x}(t)$ and $u_{+x}(t)$ respectively, such that
$S_{-x} \leq S \leq S_{+x}$.
We recognize each $S_{\pm x}$ as an integral transform of $\mathcal{D}(s, w)$
against $U_{\pm x}(s)$, and use the meromorphic continuation of $\mathcal{D}(s, w)$ to
produce bounds.

\section{Application to APs with constrained ratios}\label{sec:applications}

One of the advantages in studying the meromorphic properties of $\mathcal{D}(s,w)$ is its
flexibility in producing asymptotics for a variety of sums related to APs
$\{h,m,2m-h\}$ of squares. We demonstrate this flexibility through examples, by applying
classical Tauberian techniques to $\mathcal{D}(s, w)$.

In this section, we count primitive APs of squares in which $m \leq X$
and $h/m \leq \delta$ for some fixed $\delta \in (0, 1)$.
In Section~\ref{sec:further_applications}, we give three further applications:
counts for primitive APs with bounded maximum, with independently bounded first and second terms,
and for APs in which $mh \leq X$.

This section serves as a model for the later applications. We provide complete
details here, as this application requires the most explicit computation. The
applications in Section~\ref{sec:further_applications} are similar, but simpler.

\subsection{Statement of Result}%
\label{ssec:ratl_points}

We study sums of the form
\begin{equation}
  S(X, \delta)
  :=
 \sum_{m \leq X} \sideset{}{'}\sum_{h/m \leq \delta} r_1(m)r_1(h)r_1(2m-h),
\end{equation}
in which the `prime' denotes the restriction $(m,h)=1$. As described in
Section~\ref{ssec:equidistribution}, this sum also counts rational points on the
circle $x^2 + y^2 = 2$.

Our main result of this section is the following theorem.
\begin{theorem}\label{thm:ratl_points}
  Fix $\delta \in [0, 1]$. Then for any $\epsilon > 0$, the number of primitive APs of squares
$\{h,m,2m-h\}$ with $m \leq X$ and $(h/m) \leq \delta$ is
  \begin{equation}
  \begin{split}
    \tfrac{1}{8} S(X, \delta)
    &= \tfrac{1}{8}\sum_{m \leq X} \sideset{}{'}\sum_{h/m \leq \delta} r_1(m) r_1(h) r_1(2m - h)
    \\
    &=
    \frac{2}{\pi^2} \arcsin(\sqrt{\delta / 2}) X^{\frac{1}{2}}
    + O_\epsilon(X^{\frac{3}{8}+ \epsilon}).
  \end{split}
  \end{equation}
\end{theorem}

\begin{proof}
To prove this theorem, we define
\begin{align}
  S_{-x}(X, \delta)
  &:=
  \sideset{}{'}\sum_{m, h \geq 1} r_1(m)r_1(h)r_1(2m - h)
  u_{-x}\big(\tfrac{m}{X}\big)
  u_{-x}\big(\tfrac{h}{m\delta}\big)
  \\
  S_{+x}(X, \delta)
  &:=
  \sideset{}{'}\sum_{m, h \geq 1} r_1(m)r_1(h)r_1(2m - h)
  u_{+x}\big(\tfrac{m}{X}\big)
  u_{+x}\big(\tfrac{h}{m\delta}\big),
\end{align}
where $u_{\pm x}$ are the weight functions described in
Section~\ref{sec:weight_functions}.
By construction of $u_{\pm x}$ and the nonnegativity of the coefficients, we have the inequalities
\begin{equation}
  S_{-x}(X, \delta) \leq S(X, \delta) \leq S_{+x}(X, \delta).
\end{equation}

We recognize $S_{-x}$ and $S_{+x}$ as integral transforms of $\mathcal{D}(s-w, w)$:
\begin{equation}\label{eq:model_transform}
  S_{\pm x}(X, \delta)
  =
  \frac{1}{(2 \pi i)^2} \int_{(\sigma_w)}\int_{(\sigma_s)}
  \mathcal{D}(s - w, w) U_{\pm x}(s)U_{\pm x}(w) X^s \delta^w ds \, dw,
\end{equation}
where $\sigma_w$ and $\sigma_s$ are within the region of absolute convergence of
the multiple Dirichlet series $\mathcal{D}(s,w)$. We take $\sigma_w = \frac{1}{4}$ and
$\sigma_s = 10$ initially, which is justified by the upper bound
$r_1(h) r_1(m) r_1(2m-h) \ll 1$ and the fact that $r(2m-h) = 0$ for
$h > 2m$.

Our treatments of $S_{-x}$ and $S_{+x}$ are nearly identical.
From Theorem~\ref{thm:double_ds_decomposition}, the analysis of each integral
transform breaks up into the analysis of two pieces: the $m=0$ term, which
corresponds to the Dedekind zeta function for $\mathbb{Q}(\sqrt{2})$, and the
remainder of the discrete spectrum from~\eqref{eq:in_thm_double_ds_decomposition}.
We denote these integrals by $I_{\pm x}^{0}$ and $I_{\pm x}^{\mathrm{spec}}$,
respectively. Then we can rewrite~\eqref{eq:model_transform} as
\begin{equation}
  S_{\pm x}(X, \delta)
  =
  I_{\pm x}^{0}(10, \tfrac{1}{4}, X, \delta)
  +
  I_{\pm x}^{\mathrm{spec}}(10, \tfrac{1}{4}, X, \delta).
\end{equation}
We study $I_{\pm x}^{0}$ in Section~\ref{ssec:first_term_ratlpoints} and
$I_{\pm x}^{\mathrm{spec}}$ in Section~\ref{ssec:discrete_term_ratl_points},
culminating in the bounds from Propositions~\ref{prop:first_I_bound}
and~\ref{prop:discrete_I_bound}. When combined, these bounds give
\begin{equation}
  S_{\pm x}(X, \delta)
  =
  \frac{16}{\pi^2} \arcsin(\sqrt{\delta / 2}) X^{\frac{1}{2}}
  + O\Big( \frac{X^{\frac{1}{2}}}{x}
    + X^{\frac{1}{4}+\epsilon} x^{1 + \epsilon}
  \Big).
\end{equation}
Choosing
$x = X^{1/8}$ balances the error terms and the inequalities
$S_{-x}(X, \delta) \leq S(X, \delta) \leq S_{+x}(X, \delta)$ imply the theorem.
\end{proof}

In the remainder of this section, we provide the remaining technical details and
bounds used in the proof of this theorem.

\subsection{Principal term}\label{ssec:first_term_ratlpoints}

The primary growth in the integrals~\eqref{eq:model_transform} comes from the $m=0$ term in the decomposition~\eqref{eq:in_thm_double_ds_decomposition} for $\mathcal{D}(s,w)$.  From the identity $L(s,\eta^0)=\zeta(s)L(s,\chi)$ and the gamma duplication formula, this is
\begin{equation}
  \begin{split}
  I^{0}_{\pm x}(\sigma_s, \sigma_w, X, \delta)
  =
  \frac{2\sqrt{\pi}}{(2 \pi i)^2}
  &\int_{(\sigma_w)}
  \int_{(\sigma_s)}
  \Big(
  U_{\pm x}(s) U_{\pm x}(w) X^s (\delta/2)^w
  \\
  &\times
  \frac{%
    2^s\Gamma(s - w) \zeta^{(2)}(2s) L(2s, \chi)
  }{%
    \log(1 + \sqrt 2) \Gamma(\frac{1}{2}+s -w) \zeta^{(2)}(4s)
  } \Big) ds \, dw,
  \end{split}
\end{equation}
in which $\sigma_w = \frac{1}{4}$ and $\sigma_s = 10$ to begin.

The gamma and $L$-functions in the integrand conspire to give at most polynomial growth in vertical strips, which is counteracted by arbitrary polynomial decay in the weight functions $U_{\pm x}(s)$ and $U_{\pm x}(w)$.  We are thus free to shift lines of integration and extract residues.

Shifting the line of $s$-integration to $\sigma_s = \frac{1}{4}+\epsilon$ passes a simple pole at $s=\frac{1}{2}$ from the zeta function.  The residue at this pole can be written
\begin{align} \label{eq:MT_residue}
R_{\frac{1}{2}} :=
\frac{4X^\frac{1}{2}}{2\pi i} \int_{(\frac{1}{4})} \frac{(\delta/2)^w \Gamma(\frac{1}{2}-w) U_{\pm x}(\frac{1}{2})U_{\pm x}(w)}{\pi^\frac{3}{2} \Gamma(1-w)} \,dw.
\end{align}
The dominant growth of this integral can be explicitly evaluated using the Mellin-inversion relationship
\begin{equation}\label{eq:arcsine_identity}
  \frac{1}{2\pi i} \int_{(\frac{1}{4})}
  \frac{\Gamma(\frac{1}{2} - w)}{\Gamma(1 - w) w}
  z^w dw
  =
  \frac{2}{\sqrt \pi} \arcsin(\sqrt z)
\end{equation}
and the approximation $U_{\pm x}(w) = 1/w + O(1/x)$.  To justify our treatment of the $O(1/x)$ error term, we require a technical lemma.

\begin{lemma}\label{lem:MT_eval_IPB}
  Fix $\epsilon > 0$ and a meromorphic function $F(w)$ satisfying
  $F(w) \ll \vert \Im w \vert^{-\epsilon}$ on $\Re w=\sigma_w$.
  Define $H(z) =\frac{1}{2\pi i} \int_{(\sigma_w)} F(w) \frac{z^w}{w} \, dw$.
Then $H(z)$ is meromorphic and for $z \geq 0$, we have
  \begin{equation*}
    \frac{1}{2\pi i} \int_{(\sigma_w)} F(w) U_{\pm x}(w) z^w \, dw
    =
    H(z) + O\Big(\frac{z}{x} \sup_{\lvert s-z\rvert < \lvert z \rvert /x} \vert H'(s) \vert \Big).
  \end{equation*}
  for $x \gg 1$.
\end{lemma}

\begin{proof}
We expand $U_{\pm x}$ as an integral, integrate by parts, and then swap the
order of integration. Expanding $U_{\pm x}$ and integrating by parts gives that
\begin{align}
  &\frac{1}{2 \pi i} \int_{(\sigma_w)}
  F(w) U_{\pm x}(w) z^w dw
  =
  \frac{1}{2\pi i} \int_{(\sigma_w)}
  F(w) z^w
    \int_0^\infty u_{\pm x}(t) t^w \frac{dt}{t}
  dw
  \\
  &\quad= \frac{1}{2\pi i} \int_{(\sigma_w)}
  F(w) z^w
  \Big(
    -\frac{1}{w} \int_0^{1 + 1/x}
    u_{\pm x}'(t) t^w dt
  \Big)
  dw.
\end{align}
The $w$-integrand now decays as $O(\vert \Im w \vert^{-1-\epsilon})$, which
justifies the interchange of the order of integration.  After reordering, this becomes
\begin{align}
&- \int_0^{1 + \frac{1}{x}} u'_{\pm x}(t) \Big(\frac{1}{2\pi i} \int_{(\sigma_w)} F(w) (zt)^w \frac{dw}{w}\Big) dt = - \int_0^{1 + \frac{1}{x}} u'_{\pm x}(t) H(zt) dt.
\end{align}
Cauchy's differentiation formula implies that $H$ is meromorphic. Integration by parts presents the last integral above in the form
\begin{align}
&-u_{\pm x}(t) H(zt) \bigg\vert_{t=0}^{1+\frac{1}{x}} + \int_{0}^{1+ \frac{1}{x}} u_{\pm x}(t) zH'(zt)\, dt \\
&\quad = H(0) +\int_0^1 z H'(zt)dt + \int_1^{1+\frac{1}{x}} u_{\pm x}(t)zH'(zt)dt.
\end{align}
The first two terms give $H(z)$ and the last term gives the stated error.
\end{proof}

By combining Lemma~\ref{lem:MT_eval_IPB} and~\eqref{eq:arcsine_identity}, we
see that the residue $R_{\frac{1}{2}}$ defined in~\eqref{eq:MT_residue}
satisfies
\begin{align} \label{eq:MT_residue_bounds}
R_{\frac{1}{2}} =
\frac{16}{\pi^2} \arcsin\big(\sqrt{\delta/2}\big)X^\frac{1}{2} + O\Big(\frac{X^\frac{1}{2}}{x}\Big).
\end{align}


We still need to understand the growth of the shifted integral. It will turn out
that this is not the primary obstruction to a better result, so we can use a
coarse bound.
\begin{lemma}\label{lem:MT_shifted_integral}
  With the notation as above, we have that
  \begin{equation}
  I^{0}_{\pm x}(\tfrac{1}{4} + \epsilon, \tfrac{1}{4}, X, \delta)
  \ll X^{\frac{1}{4} + \epsilon} x^{\frac{1}{2} + 2\epsilon}.
  \end{equation}
\end{lemma}

\begin{proof}

We bound the integrand in absolute value, approximate the gamma functions by
Stirling's formula, and bound the $L$-functions by the convexity bound. We also
note the classical inequality $1/\zeta(1 + \epsilon + it) \ll \log t$.

We can extract the factor $X^{1/4 + \epsilon}$ immediately. Within the integral,
the exponential growth from the gamma functions cancels out. The total
polynomial growth of the gamma functions and $L$-functions is of size $(1 +
\lvert \Im s - \Im w \rvert)^{-1/2} (1 + \lvert \Im s \rvert)^{1/2 + \epsilon}$.
Absolute convergence of the integral follows from
$U_{\pm x}(s) \ll x^{1/2 + \epsilon}(1 + \lvert s \rvert)^{-(3/2 +
\epsilon)}$ and $U_{\pm x}(w) \ll x^{\epsilon}(1 + \lvert w \rvert)^{-(1 +
\epsilon)}$ as given in~\eqref{eq:weight_Us_bound}. Combining gives the proof.
\end{proof}

We thus have that
\begin{equation}
  I^{0}_{\pm x}(10, \tfrac{1}{4}, X, \delta)
  =
  R_{\frac{1}{2}}
  +
  I^{0}_{\pm x}(\tfrac{1}{4} + \epsilon, \tfrac{1}{4}, X, \delta),
\end{equation}
where $R_{1/2}$ is the residue in~\eqref{eq:MT_residue}. Bounds for the residue are
given in~\eqref{eq:MT_residue_bounds} and the
size of the shifted integral is bounded in Lemma~\ref{lem:MT_shifted_integral}.
Assembling these bounds together proves the following.

\begin{proposition}\label{prop:first_I_bound}
  For any fixed $\epsilon > 0$ and $\delta \in [0,1]$, we have
  \begin{equation}
  \begin{split}
    I^{0}_{\pm x}(10, \tfrac{1}{4}, X, \delta)
    &=
    \frac{16}{\pi^2} \arcsin(\sqrt{\delta / 2}) X^{\frac{1}{2}}
    + O\Big(\frac{X^{\frac{1}{2}}}{x}\Big)
    + O(X^{\frac{1}{4} + \epsilon} x^{\frac{1}{2} + \epsilon}).
  \end{split}
  \end{equation}
\end{proposition}

\subsection{Discrete term}\label{ssec:discrete_term_ratl_points}

Within the integrals~\eqref{eq:model_transform}, the term coming from the
discrete sum in the decomposition~\eqref{eq:in_thm_double_ds_decomposition}
takes the form
\begin{equation}
  \begin{split}
  I^{\mathrm{spec}}_{\pm x}&(\sigma_s, \sigma_w, X, \delta)
  \\
  &=
  \frac{1}{(2 \pi i)^2}
  \int_{(\sigma_w)}
  \int_{(\sigma_s)}
  \sum_{m \neq 0}
  \Big(
  U_{\pm x}(s) U_{\pm x}(w) X^s \delta^w
  2^{3s -3w} (1-2^{-2s})
  \\
  &\qquad\times
  \frac{\Gamma(s - w - it_m) \Gamma(s - w + it_m)}
       {\Gamma(2s -2 w)\log(1+\sqrt{2})}
 \frac{(-1)^mL(2s,\eta^{2m})}{\zeta^{(2)}(4s)}  \Big) ds \, dw,
  \end{split}
\end{equation}
where initially $\sigma_w = \frac{1}{4}$ and $\sigma_s = 10$.
Since $L(2s,\eta^{2m})$ is entire for $m \neq 0$, we have
\begin{equation}
	I_{\pm x}^{\mathrm{spec}}(10,\tfrac{1}{4},X,\delta)
	= I_{\pm x}^{\mathrm{spec}}(\tfrac{1}{4}+\epsilon,\tfrac{1}{4},X,\delta).
\end{equation}
To justify this contour shift and then bound $I_{\pm x}^{\mathrm{spec}}(\frac{1}{4}+\epsilon,\frac{1}{4},X,\delta)$,
we require a technical lemma to handle growth in the $m$-sum.

\begin{lemma} \label{lem:gamma_ratio_lemma}
On the lines $\Re z = x \in (0,\frac{1}{2})$ and $\Re s = \frac{1}{2}+\epsilon$, we have
\begin{align} \label{eq:gamma_ratio_lemma}
\sum_{m \neq 0} \bigg\vert \frac{\Gamma(z+it_m)\Gamma(z-it_m)L(s,\eta^{2m})}{\Gamma(2z)}\bigg\vert \ll (1+\vert z \vert + \vert z \vert^\frac{1}{2} \vert s \vert^{\frac{1}{2}}).
\end{align}
\end{lemma}

\begin{proof} The convexity estimate $L(s,\eta^{2m}) \ll (1+\vert s + it_m \vert)^{\frac{1}{4}}(1+\vert s - it_m \vert)^{\frac{1}{4}}$ on the line $\Re s = \frac{1}{2}+\epsilon$ and uniform estimates for Stirling's approximation in the right half-plane suffices to bound the sum in~\eqref{eq:gamma_ratio_lemma} by
\[\sum_{m \neq 0} \frac{\vert x+iy\vert^{\frac{1}{2}-2x} (1+\vert s \vert^{\frac{1}{2}} + \vert t_m \vert^{\frac{1}{2}})}{(\vert x+iy+it_m \vert \cdot \vert x+iy-it_m \vert)^{\frac{1}{2}-x}} \mathrm{exp}\big( -\pi \max(\vert t_m \vert, \vert y \vert) + \pi \vert y \vert\big).\]
Since $\vert x+iy \pm it \vert \geq \max(x,\vert y \pm t \vert)$, we have
\begin{align*}
&\frac{\vert x+iy\vert^{\frac{1}{2}-2x}}{(\vert x+iy+it_m \vert \cdot \vert x+iy-it_m \vert)^{\frac{1}{2}-x}} \\
&\qquad \qquad\qquad \ll \vert x+iy\vert^{\frac{1}{2}-2x} \min\bigg(\frac{1}{\vert x \vert^{1-2x}}, \frac{1}{\vert y+t_m \vert^{\frac{1}{2}-x} \vert y-t_m \vert^{\frac{1}{2}-x}}\bigg).
\end{align*}
The first term in the minimum is $O(1)$ and the latter may be bounded in cases depending on the signs and relative sizes of $t_m$ and $y$.

The contribution in the case $0<t_m < y$ is
\[\ll \sum_{0<t_m< y} \frac{\vert x+iy\vert^{\frac{1}{2}-2x}(1+\vert s \vert^{\frac{1}{2}} + \vert t_m \vert^{\frac{1}{2}})}{\vert 1+y \vert^{\frac{1}{2}-x} \vert y-t_m \vert^{\frac{1}{2}-x}}
\ll 1+ y+y^{\frac{1}{2}} \vert s \vert^{\frac{1}{2}},\]
as can be seen via integral comparison. The contribution from $0<y<t_m$ experiences exponential decay and is smaller, no larger than
\[\sum_{y<t_m} \frac{\vert x+iy\vert^{\frac{1}{2}-2x}(1+\vert s \vert^{\frac{1}{2}} + \vert t_m \vert^\frac{1}{2})}{\vert 1+t_m \vert^{\frac{1}{2}-x} \vert t_m-y\vert^{\frac{1}{2}-x}}e^{-\pi (t_m-y)}\ll 1 + y^{\frac{1}{2}-x} + \frac{\vert s \vert^{\frac{1}{2}}}{(1+y)^{x}}.\]
The cases in which $y$ and $t_m$ differ in sign are analogous.
\end{proof}

The bound in Lemma~\ref{lem:gamma_ratio_lemma} implies that
\begin{align} \label{eq:I-spec-bound}
	I_{\pm x}^{\mathrm{spec}}(\tfrac{1}{4}+\epsilon, \tfrac{1}{4},X,\delta)
	\ll X^{\frac{1}{4}+\epsilon}\int_{(\frac{1}{4})}\int_{(\frac{1}{4}+\epsilon)}
	&	\vert U_{\pm x}(s) U_{\pm x}(w) \vert \\
	\times & (1+\vert s-w\vert + \vert s \vert^{\frac{1}{2}} \vert s-w\vert^{\frac{1}{2}}) \, ds dw.
\end{align}
In the region $\vert s \vert > \vert w \vert$, the bounds
$U_{\pm x}(s) \ll x^{1+\epsilon}/\vert s \vert^{2+\epsilon}$ and $U_{\pm x}(w) \ll x^\epsilon/\vert w \vert^{1+\epsilon}$ ensure convergence of the integral. In the region $\vert s \vert < \vert w \vert$, we instead apply $U_{\pm x}(s) \ll x^{\epsilon}/\vert s \vert^{1+\epsilon}$ and $U_{\pm x}(w) \ll x^{1+\epsilon}/\vert w \vert^{2+\epsilon}$.
By applying these bounds to~\eqref{eq:I-spec-bound}, we produce the following proposition.

\begin{proposition}\label{prop:discrete_I_bound}
For any $\epsilon > 0$ and $\delta \in [0, 1]$, we have
\begin{equation}
  I_{\pm x}^{\mathrm{spec}}(10, \tfrac{1}{4}, X, \delta)
  \ll_{\epsilon} X^{\frac{1}{4}+\epsilon} x^{1+\epsilon}.
\end{equation}
\end{proposition}

\begin{remark}
Lemma~\ref{lem:gamma_ratio_lemma} and thereafter Proposition~\ref{prop:discrete_I_bound}
may be improved by applying subconvexity results for the Hecke $L$-functions
$L(s,\eta^{2m})$ (such as~\cite{jutila2005uniform} or~\cite{Wu19}).
The estimate $L(s,\eta^{2m}) \ll (1+\vert s +it_m\vert)^\alpha(1+\vert s-it_m \vert)^\alpha$
on the line $\Re s = \frac{1}{2}+\epsilon$ implies that $I_{\pm x}^\mathrm{spec}$ in
Proposition~\ref{prop:discrete_I_bound} is $O(X^{1/4+\epsilon} x^{1/2+2\alpha+\epsilon})$ and
that Theorem~\ref{thm:ratl_points} holds in the form
\[
\tfrac{1}{8}S(X,\delta) = \frac{2}{\pi^2} \arcsin(\sqrt{\delta/2}) X^{\frac{1}{2}}
+ O_\epsilon\big(X^{\frac{1}{2}-\frac{1}{6+8\alpha}+\epsilon}\big).
\]
Under the Lindel\"of Hypothesis, we have $\alpha = 0$.
The same improvements hold in the conclusions of Theorems~\ref{thm:boundedmax},~\ref{thm:naturalXY}, and~\ref{thm:firsttwoX}.
\end{remark}

\section{Further Applications}\label{sec:further_applications}

This section contains three further applications of the meromorphic description of
$\mathcal{D}(s, w)$. In Section~\ref{ssec:boundedmax}, we count primitive APs with
bounded maximum; in Section~\ref{ssec:naturalXY}, we count primitive APs with
individually bounded first and second terms; and in Section~\ref{ssec:firsttwoX},
we count primitive APs whose first two terms have bounded product. Many of the
technical details are similar to those in Section~\ref{sec:applications}, so we
describe only those portions that lead to the main terms and dominant error terms
in the asymptotics.

\subsection{APs with bounded maximum}
\label{ssec:boundedmax}

We count the number of primitive APs of squares $\{h,m,2m-h\}$ with bounded maximum,
In particular, we produce asymptotics for sums of the form
\[T(X) := \sum_{h \leq X} \sideset{}{'}\sum_{m \leq h} r_1(h)r_1(m)r_1(2m-h),\]
in which the prime indicates the restriction to $(m,h)=1$.  Note that such APs are
arranged with $\{h,m,2m-h\}$ decreasing; this does not affect our results but simplifies
notation.

Our primary theorem in this section is the following.

\begin{theorem} \label{thm:boundedmax}
The number of primitive APs of squares with largest term at most $X$ is
\[\tfrac{1}{8}T(X) = \frac{\sqrt{2}}{\pi^2} \log(1+\sqrt{2})X^{\frac{1}{2}} + O_\epsilon \big(X^{\frac{3}{8}+\epsilon}\big)\]
for any $\epsilon > 0$.
\end{theorem}

To prove this theorem, we define $T_{\pm x}$ as
\[T_{\pm x}(X) := \sideset{}{'}\sum_{m,h \geq 1} r_1(h)r_1(m)r_1(2m-h) u_{\pm x}(\tfrac{m}{h}) u_{\pm x}(\tfrac{h}{X})\]
and recognize these sums as the integral transforms
\begin{align} \label{eq:boundedmax_model_transform}
T_{\pm x}(X) = \frac{1}{(2\pi i)^2} \int_{(\sigma_w)}\int_{(\sigma_s)} \mathcal{D}(s,w-s) X^w U_{\pm x}(s) U_{\pm x}(w) \, dsdw
\end{align}
for sufficiently large $\sigma_w$ and $\sigma_s$. We take $\sigma_w =10$ and $\sigma_s = \frac{1}{4}$ to begin.

The proof of Theorem~\ref{thm:boundedmax} is very similar to the proof of
Theorem~\ref{thm:ratl_points}. In particular, our analysis once again follows
the decomposition of $\mathcal{D}(s,w)$ given in Theorem~\ref{thm:double_ds_decomposition}.
The main term in Theorem~\ref{thm:boundedmax} comes from the $m=0$ term in the
expansion~\eqref{eq:in_thm_double_ds_decomposition} and the error term results
from balancing against estimates in the $m\neq 0$ component
from~\eqref{eq:in_thm_double_ds_decomposition}.  We therefore sketch only these
parts of the proof.

\subsection*{Principal term}

The contribution of the $m=0$ term in the expansion~\eqref{eq:in_thm_double_ds_decomposition}
towards $T_{\pm x}(X)$ takes the form
\begin{align}
&I_{\pm x}^0(\sigma_s,\sigma_w,X)  = \\
&\quad  \frac{1}{(2\pi i)^2} \!\int_{(\sigma_w)} \!\int_{(\sigma_s)}\!\!\!
	\frac{2^{3s} \zeta^{(2)}(2w) L(2w,\chi)\Gamma(s)^2 X^w U_{\pm x}(s) U_{\pm x}(w)}{\zeta^{(2)}(4w) \log(1+\sqrt{2}) \Gamma(2s)} dsdw,
\end{align}
which conveniently decouples into independent $s$ and $w$ integrals.

To handle the
$s$-integral, note by contour shifting the Mellin pair identity
\[\frac{1}{2\pi i} \int_{(\sigma_s)} \frac{2^{3s} \Gamma(s)^2}{\Gamma(2s)} z^s\, \frac{ds}{s} = 4\,\mathrm{arctanh}\Big(\sqrt{1-\tfrac{1}{2z}}\Big),\]
which holds for $\Re s > 0$ and $\vert z \vert > \frac{1}{2}$.
It follows by Lemma~\ref{lem:MT_eval_IPB} that
\[\frac{1}{2\pi i}\int_{(\sigma_s)} \frac{2^{3s} \Gamma(s)^2 U_{\pm x}(s)}{\Gamma(2s)} ds
	= 4\log(1+\sqrt{2}) + O(1/x),\]
in which we've used that $\mathrm{arctanh}(1/\sqrt{2}) = \log(1+\sqrt{2})$.

To estimate the integral in $w$, we shift the line of integration to
$\sigma_w = \frac{1}{4}+\epsilon$, passing a simple pole at $w=\frac{1}{2}$
and extracting the residue
\[\frac{\sqrt{2}U_{\pm x}(\tfrac{1}{2})}{\pi^2} X^{\frac{1}{2}}  = \frac{2\sqrt{2}}{\pi^2} X^{\frac{1}{2}} + O\Big(\frac{X^\frac{1}{2}}{x}\Big).\]
The shifted $w$-integral is easily seen to be $O(X^{1/4+\epsilon} x^{1/2+\epsilon})$, hence
\begin{align} \label{eq:boundedmax_principal}
I_{\pm x}^0(\tfrac{1}{4},10,X) = \frac{8\sqrt{2}}{\pi^2} \log(1+\sqrt{2}) X^{\frac{1}{2}}
 + O\Big(\frac{X^{\frac{1}{2}}}{x} + X^{\frac{1}{4}+\epsilon} x^{\frac{1}{2}+\epsilon}\Big).
\end{align}

\subsection*{Discrete term}

Within the integrals~\eqref{eq:boundedmax_model_transform}, the contribution of the terms with $m\neq 0$ in~\eqref{eq:in_thm_double_ds_decomposition}
takes the form
\begin{align}
I_{\pm x}^\mathrm{spec}(\sigma_s,\sigma_w,X)
&= \frac{1}{(2\pi i)^2} \int_{(\sigma_w)} \int_{(\sigma_s)}
	\!\frac{2^{3s} (1-2^{-2w}) U_{\pm x}(w) U_{\pm x}(s) X^w }{\zeta^{(2)}(4w) \log(1+\sqrt{2})} \\
	&\quad \times \sum_{m \neq 0} \frac{\Gamma(s+it_m)\Gamma(s-it_m)}{\Gamma(2s)}
	(-1)^m L(2w,\eta^{2m})\,  dsdw,
\end{align}
in which $\sigma_w = 10$ and $\sigma_s = \frac{1}{4}$ initially.  Since $L(w,\eta^{2m})$ is entire for $m \neq 0$, we have $I_{\pm x}^\mathrm{spec}(\frac{1}{4},10,X) = I_{\pm x}^\mathrm{spec}(\frac{1}{4},\frac{1}{4}+\epsilon, X)$. Lemma~\ref{lem:gamma_ratio_lemma} implies that
\[I_{\pm x}^\mathrm{spec}(\tfrac{1}{4},\tfrac{1}{4}+\epsilon, X)
	\ll X^{\frac{1}{4}+\epsilon}\! \int_{(\sigma_w)} \int_{(\sigma_s)} \vert U_{\pm x}(s)U_{\pm x}(w) \vert (\vert s \vert + \vert s \vert^{\frac{1}{2}} \vert w \vert^{\frac{1}{2}}) dsdw.\]
Various estimates for $U_{\pm x}(s)$ and $U_{\pm x}(w)$ given by~\eqref{eq:weight_Us_bound} then imply
\[I_{\pm x}^\mathrm{spec}(\tfrac{1}{4},10,X) = O\big(X^{\frac{1}{4}+\epsilon} x^{1+2\epsilon}\big).\]

\begin{proof}[Proof of Theorem~\ref{thm:boundedmax}]
We have $ T_{\pm x}(X)
  =
  I_{\pm x}^{0}(\tfrac{1}{4}, 10, X)
  +
  I_{\pm x}^{\mathrm{spec}}(\tfrac{1}{4}, 10, X)$.
We have shown that $I_{\pm x}^{0}$ contributes the
term~\eqref{eq:boundedmax_principal} and $I_{\pm x}^{\mathrm{spec}}$
contributes the error $O(X^{\frac{1}{4}+\epsilon} x^{1+2\epsilon})$, so that
\begin{equation}
\begin{split}
  T_{\pm x}(X, Y)
  &=
  \frac{8\sqrt{2}}{\pi^2} X^{\frac{1}{2}}
	+ O\Big(\frac{X^{\frac{1}{2}}}{x} + X^{\frac{1}{4}+\epsilon} x^{1+2\epsilon}\Big).
\end{split}
\end{equation}
The theorem now follows by choosing $x=X^{1/8}$ and applying the inequalities $T_{-x}(X) \leq T(X)
\leq T_{+x}(X)$.
\end{proof}

\subsection{APs with individually bounded first terms}%
\label{ssec:naturalXY}

We count the number of primitive APs of squares $\{h, m, 2m - h\}$ such that $m \leq X$
and $h \leq Y$, with $Y \leq X$. In particular, we produce asymptotics for sums
of the form
\begin{equation}
  S(X, Y)
  :=
  \sum_{m \leq X} \sideset{}{'}\sum_{h \leq Y} r_1(h)r_1(m)r_1(2m-h),
\end{equation}
in which the prime indicates the restriction to $(m,h)=1$.

\begin{remark}
The sum $S(X,Y)/8$ double-counts those APs in which both $h\leq Y$ and $2m-h \leq Y$.
For example, both $\{1, 25, 49\}$ and $\{49, 25, 1\}$ would be counted if $X = Y
= 50$. We can compensate for this double-counting by removing those ``reversed'' APs
whose maximum is the first element, i.e. those APs with maximum at most $Y$.
The number of such APs is given by the quantity $T(Y)$ of Theorem~\ref{thm:boundedmax},
so that $(S(X,Y)-T(Y))/8$ counts primitive APs with first term at most $Y$ and center at most $X$.
\end{remark}

Our primary theorem of this section is the following.

\begin{theorem}\label{thm:naturalXY}
Suppose that $Y \leq X$. Then, for any $\epsilon > 0$, the number of primitive APs of
squares $\{h,m,2m-h\}$ with $h \leq Y$ and $m \leq X$ is
\[\tfrac{1}{8}S(X,Y) - \tfrac{1}{8} T(Y)
= \frac{1}{\sqrt{2} \pi^2} Y^{\frac{1}{2}} \log \big(X/Y\big)
+  c\, Y^{\frac{1}{2}} + O_\epsilon\big(X^\epsilon Y^{\frac{3}{8}+\epsilon}\big),\]
in which $c= \sqrt{2}(1+\frac{3}{2}\log 2 - \log(1+\sqrt{2}))/\pi^2$.
\end{theorem}

To prove this theorem, we define $S_{+x}(X,Y)$ and $S_{-x}(X,Y)$ as
\begin{align}
  S_{\pm x}(X, Y)
  &:=
  \sideset{}{'}\sum_{m, h \geq 1} r_1(m)r_1(h)r_1(2m - h)
  u_{\pm x}\big(\tfrac{m}{X}\big)
  u_{\pm x}\big(\tfrac{h}{Y}\big) \\
\label{eq:naturalXY_model_transform}
 & =
  \frac{1}{(2 \pi i)^2} \int_{(\sigma_w)}\int_{(\sigma_s)}
  \mathcal{D}(s, w) U_{\pm x}(s)U_{\pm x}(w) X^s Y^w ds \, dw
\end{align}
for sufficiently large $\sigma_w$ and $\sigma_s$. We take $\sigma_w =
\frac{1}{4}$ and $\sigma_s = 10$ initially.

Our analysis of $S_{\pm x}(X,Y)$ follows the decomposition of $\mathcal{D}(s,w)$
given in Theorem~\ref{thm:double_ds_decomposition}. As in Theorems~\ref{thm:ratl_points}
and~\ref{thm:boundedmax}, the main term comes from the $m=0$ term
in~\eqref{eq:in_thm_double_ds_decomposition} and the $m \neq 0$ term contributes
error terms used for optimizing the parameter $x$.

\subsection*{Principal term}

Simplifying the $m=0$ term of~\eqref{eq:in_thm_double_ds_decomposition} as in Section~\ref{ssec:first_term_ratlpoints}, we see that its contribution to $S_{\pm x}(X,Y)$
takes the form
\begin{equation}
  \begin{split}
  I^{0}_{\pm x}(\sigma_s, \sigma_w, X, Y)
  &=
  \frac{1}{(2 \pi i)^2}
  \int_{(\sigma_w)}
  \int_{(\sigma_s)}
  \Big(
  U_{\pm x}(s) U_{\pm x}(w) X^s Y^w
  \\
  &\qquad\times
\frac{2^{3s}\zeta^{(2)}(2s+2w)L(2s+2w,\chi)}{\zeta^{(2)}(4s+4w) \log(1+\sqrt{2})\Gamma(2s)}
    \Gamma(s)^2
 \Big) ds \, dw,
  \end{split}
\end{equation}
where initially $\sigma_w = \frac{1}{4}$ and $\sigma_s = 10$. To extract the
main term, we will
\begin{enumerate}
\item change variables to disentangle $s$ and $w$ in the leading poles,
\item shift the line of $s$ integration to the left, passing double poles, and
\item shift the line of $w$ integration \emph{within the residues extracted from} (2) to
      the right, passing triple poles.
\end{enumerate}
The shifted integrals in the first term do not contribute leading terms in the final asymptotic,
so we only consider the residues.

%

Changing variables $s \mapsto s - w$ shows that
\begin{equation}
  \begin{split}
  I^{0}_{\pm x}(10, \tfrac{1}{4}, X, Y)
  &=
  \frac{1}{(2 \pi i)^2}
  \int_{(\frac{1}{4})}
  \int_{(10 + \frac{1}{4})}
  \Big(
  U_{\pm x}(s - w) U_{\pm x}(w) X^{s-w} Y^w
  \\
  &\qquad\times
  \frac{%
    2^{3s-3w} \Gamma(s-w)^2 \zeta^{(2)}(2s) L(2s, \chi)
  }{%
    \zeta^{(2)}(4s)\log(1 + \sqrt 2) \Gamma(2s -2w)
  } \Big) ds \, dw.
  \end{split}
\end{equation}
This integrand matches the integrand within $I^0_{\pm x}$ from
Section~\ref{ssec:ratl_points}, except that this has $U_{\pm x}(s-w)X^{s-w}$
instead of $U_{\pm x}(s)X^s$. Shifting the line of
$s$-integration to $\sigma_s = \frac{1}{4} + \epsilon$ passes a simple pole with residue
%
%
\begin{align}
\frac{4X^{\frac{1}{2}}}{2\pi i} \int_{(\frac{1}{4})} \frac{\Gamma(\frac{1}{2}-w)^2 U_{\pm x}(\frac{1}{2}-w)U_{\pm x}(w)}{\pi^2 \Gamma(1-2w)}\Big(\frac{Y}{8X}\Big)^w \,dw.
\end{align}


As $Y < X$, this residue is
minimized by shifting $w$ to the right.  We shift $w$ to $\sigma_w = 1-\epsilon$,
extracting a single residue from a pole of order two at $w=\frac{1}{2}$.
The negation of this residue takes the form
\begin{equation}\label{eq:naturalXY_first_residue}
  \frac{4\sqrt{2}}{\pi^2} \log(X/Y) Y^{\frac{1}{2}}
 + \frac{8\sqrt{2} + 12 \sqrt{2}\log 2}{\pi^2}\, Y^\frac{1}{2}
 + O\Big(\frac{\log(X/Y)Y^{\frac{1}{2}} + Y^\frac{1}{2}}{x}\Big).
\end{equation}
The main term above creates the main term in our final asymptotic.

\subsection*{Discrete term}

Within the integrals~\eqref{eq:naturalXY_model_transform}, the term coming from the
$m\neq 0$ part of the decomposition~\eqref{eq:in_thm_double_ds_decomposition}
takes the form
\begin{equation}
  \begin{split}
  I^{\mathrm{spec}}_{\pm x}&(\sigma_s, \sigma_w, X, Y)
  \\
  &=
  \frac{1}{(2 \pi i)^2}
  \int_{(\sigma_w)}
  \int_{(\sigma_s)}
  \sum_{m \neq 0}
  \Big(
  U_{\pm x}(s) U_{\pm x}(w) X^s Y^w
  2^{3s} (1-2^{-2s-2w})
  \\
  &\qquad\times
  \frac{\Gamma(s  - it_m) \Gamma(s  + it_m)}
       {\Gamma(2s)\log(1+\sqrt{2})}
 \frac{(-1)^m L(2s+2w,\eta^{2m})}{\zeta^{(2)}(4s+4w)}  \Big) ds \, dw,
  \end{split}
\end{equation}
where initially $\sigma_w = \frac{1}{4}$ and $\sigma_s = 10$.  Since $L(s,\eta^{2m})$ is
entire for $m\neq 0$, we have $I_{\pm x}^\mathrm{spec}(10,\frac{1}{4},X,Y) = I_{\pm x}^\mathrm{spec}(\epsilon,\frac{1}{4},X,Y)$.
Lemma~\ref{lem:gamma_ratio_lemma} then implies that
\[I_{\pm x}^\mathrm{spec}(\epsilon,\tfrac{1}{4},X,Y) \ll X^\epsilon Y^{\frac{1}{4}} \!\int_{(\frac{1}{4})} \int_{(\epsilon)} \vert U_{\pm x}(s)U_{\pm x}(w) \vert (\vert s \vert + \vert s \vert^\frac{1}{2} \vert s+w \vert^\frac{1}{2}) \,dsdw,\]
hence $I_{\pm x}^\mathrm{spec}(\epsilon,\frac{1}{4}, X,Y) \ll X^\epsilon Y^\frac{1}{4} x^{1+\epsilon}$ using familiar bounds for $U_{\pm x}$.

\begin{proof}[Proof of Theorem~\ref{thm:naturalXY}]
We have the equality
\begin{equation}
  S_{\pm x}(X, Y)
  =
  I_{\pm x}^{0}(10, \tfrac{1}{4}, X, Y)
  +
  I_{\pm x}^{\mathrm{spec}}(10, \tfrac{1}{4}, X, Y)
\end{equation}
We have shown that $I_{\pm x}^{0}$ contributes the
term~\eqref{eq:naturalXY_first_residue} and $I_{\pm x}^{\mathrm{spec}}$
contributes the error $O(Y^\frac{1}{4} X^\epsilon x^{1+\epsilon})$, so that
\begin{equation}
\begin{split}
  S_{\pm x}(X, Y)
  &=
  \frac{4\sqrt{2}}{\pi^2} \log(X/Y) Y^{\frac{1}{2}}
 + \frac{8\sqrt{2}+12\sqrt{2} \log 2}{\pi^2}\, Y^\frac{1}{2} \\
  &\qquad + O\Big(\frac{\log(X/Y)Y^{\frac{1}{2}} + Y^\frac{1}{2}}{x}\Big)
  +
  O(Y^\frac{1}{4} X^\epsilon x^{1+\epsilon})
\end{split}
\end{equation}
The theorem follows by choosing $x=Y^{1/8}$, noting $S_{-x}(X, Y) \leq S(X, Y)
\leq S_{+x}(X, Y)$, and subtracting the estimate for $T(Y)$ given in Theorem~\ref{thm:boundedmax}.
\end{proof}

\subsection{APs with bounded products of first two terms}%
\label{ssec:firsttwoX}

Finally, we use $\mathcal{D}(s, w)$ to count the number of primitive APs of squares $\{h, m, 2m -
h\}$ such that $hm \leq X$ by producing asymptotics for the sum
\begin{equation}
  S(X) := \sideset{}{'}\sum_{\substack{hm \leq X^2 \\ h \leq m}} r_1(h) r_1(m) r_1(2m - h),
\end{equation}
in which the prime denotes the restriction $(m,h)=1$.

Our primary result of this section is the following theorem.

\begin{theorem}\label{thm:firsttwoX}
For any $\epsilon > 0$, the number of primitive APs of squares $\{h,m,2m-h\}$ with $hm \leq X$ is
  \begin{equation}
    \tfrac{1}{8} S(X)
    =
    \frac{2\sqrt{2}}{\pi^2} \, _2F_1(\tfrac{1}{4},\tfrac{1}{2},\tfrac{5}{4},\tfrac{1}{2}) X^{\frac{1}{2}}
   + O_\epsilon\big(X^{\frac{3}{8} + \epsilon}\big).
  \end{equation}
\end{theorem}

To prove this theorem, we define $S_{+x}(X)$ and $S_{-x}(X)$ as
\begin{align}
  S_{\pm x}(X)
  &:=
  \sideset{}{'}\sum_{m, h \geq 1} r_1(m)r_1(h)r_1(2m - h)
  u_{\pm x} \big(\tfrac{h}{m}\big)
  u_{\pm x}\big(\tfrac{hm}{X^2}\big) \label{eq:firsttwoX_model_transform} \\
  &=
  \frac{1}{(2 \pi i)^2} \int_{(\sigma_s)} \int_{(\sigma_w)}
  \mathcal{D}(s-w, s+w) U_{\pm x}(s) U_{\pm x}(w) X^{2s}\, dw \, ds
\end{align}
for sufficiently large $\sigma_s$. We take $\sigma_s = 10$ and $\sigma_w = \frac{1}{8}$ initially.

The proof of Theorem~\ref{thm:firsttwoX} again resembles the proof of
Theorem~\ref{thm:ratl_points}. As in our previous applications, we break our
proof into two parts, following the decomposition into $m=0$ and $m\neq 0$ from
Theorem~\ref{thm:double_ds_decomposition}.

\subsection*{Principal term}

The main term in the integrals~\eqref{eq:firsttwoX_model_transform} comes from the $m=0$
term in the decomposition~\eqref{eq:in_thm_double_ds_decomposition} and takes the form
\begin{align*}
  I^{0}_{\pm x}(\sigma_s,\sigma_w, X)
  =
  \frac{1}{(2\pi i)^2} \int_{(\sigma_s)} \int_{(\sigma_w)} &\frac{2^{3s-3w} \zeta^{(2)}(4s) L(4s,\chi) X^{2s}}{\zeta^{(2)}(8s) \log(1+\sqrt{2}) \Gamma(2s-2w)} \\
 &\qquad \times \Gamma(s-w)^2 U_{\pm x}(s) U_{\pm x}(w) \, dw ds,
\end{align*}
where initially $\sigma_s = 10$.  Shifting the line of $s$-integration to $\sigma_s = \frac{1}{8}+\epsilon$ passes a simple pole at $s=\frac{1}{4}$ and extracts a residue of the form
\begin{align} \label{eq:firsttwoX_residue}
\frac{X^{\frac{1}{2}}}{2\pi i}\int_{(\frac{1}{8})} & \frac{\Gamma(\frac{1}{4}-w)^2 U_{\pm x}(\frac{1}{4})U_{\pm x}(w)}{2^{3w-\frac{1}{4}} \pi^2 \Gamma(\frac{1}{2}-2w)} dw,
\end{align}
To determine the growth of this integral, we apply the integral identity
\begin{equation} \label{eq:hypergeometric-identity}
	\frac{1}{2\pi i} \int_{(\frac{1}{8})}
	\frac{\Gamma(\frac{1}{4}-w)^2}{w \Gamma(\frac{1}{2}-2w)} z^w \, dw
	= 8z^{1/4}\, _2F_1(\tfrac{1}{4}, \tfrac{1}{2}, \tfrac{5}{4}, 4z),
\end{equation}
valid for $\vert z \vert <\frac{1}{4}$, in the case $z= \frac{1}{8}$.
To establish~\eqref{eq:hypergeometric-identity}, we note by Stirling's approximation
that the integrand is $O(\vert 4z \vert^{\Re w} \vert w \vert^{-3/2})$. The assumption
$\vert z \vert < \frac{1}{4}$ implies that the contour $\Re w = \frac{1}{8}$ may be
shifted far to the right, where it vanishes in the limit. This shift extracts residues
from the integral at points $w= \frac{1}{4} + \ell$, for integers $\ell \geq 0$, and
the sum of these residues equals
\begin{align}
	& \sum_{\ell = 0}^\infty \Res_{w=\frac{1}{4}+\ell}
			\frac{\Gamma(\frac{1}{4}-w)^2}{w \Gamma(\frac{1}{2}-2w)} z^w
	= \sum_{\ell = 0}^\infty \Res_{w=\frac{1}{4}+\ell}
			\frac{(4z)^w \sqrt{2\pi} \Gamma(\frac{1}{4}-w)}{w \Gamma(\frac{3}{4}-w)} \\
	& \quad
	= \sqrt{2\pi} (4z)^{\frac{1}{4}} \sum_{\ell = 0}^\infty
			\frac{(-4z)^{\ell} \Gamma(\ell + \frac{1}{4})}
					{\Gamma(\ell + \frac{5}{4}) \Gamma(\frac{1}{2}-\ell) \ell !}
	= \frac{2z^{\frac{1}{4}}}{\sqrt{\pi}}
			\sum_{\ell = 0}^\infty \frac{(4z)^{\ell}
				\Gamma(\ell + \frac{1}{4})\Gamma(\ell+\frac{1}{2})}{\Gamma(\ell + \frac{5}{4}) \ell !} \\
	& \quad
	= \frac{2 z^{\frac{1}{4}}}{\sqrt{\pi}} \frac{\Gamma(\frac{1}{4})\Gamma(\frac{1}{2})}{\Gamma(\frac{5}{4})}
		{}_2F_1(\tfrac{5}{4}, \tfrac{1}{2}, \tfrac{1}{4}, 4z)
	= 8 z^{\frac{1}{4}} {}_2F_1(\tfrac{5}{4}, \tfrac{1}{2}, \tfrac{1}{4}, 4z),
\end{align}
in which we've used the gamma duplication formula and the reflection formula $\Gamma(\frac{1}{2}-z)\Gamma(\frac{1}{2}+z) = \pi \sec(\pi z)$ to simplify.

Via Lemma~\ref{lem:MT_eval_IPB} and~\eqref{eq:hypergeometric-identity}, the residue
integral~\eqref{eq:firsttwoX_residue} is
\begin{align} \label{eq:firsttwoX_first_MT}
\frac{16\sqrt{2}}{\pi^2} \, _2F_1(\tfrac{1}{4},\tfrac{1}{2},\tfrac{5}{4},\tfrac{1}{2}) X^\frac{1}{2} + O\Big(\frac{X^{\frac{1}{2}}}{x}\Big).
\end{align}
The shifted double integral $I_{\pm x}^0(\frac{1}{8}+\epsilon, \frac{1}{8}, X)$ is $O(X^{\frac{1}{4}+2\epsilon} x^{\epsilon})$, which will be non-dominant.

\subsection*{Discrete term}

Within the integrals~\eqref{eq:firsttwoX_model_transform}, the term coming from
the $m \neq 0$ term in the
decomposition~\eqref{eq:in_thm_double_ds_decomposition} takes the form
\begin{align*}
I_{\pm x}^{\text{spec}}(\sigma_s,\sigma_w,X) = &\frac{1}{(2\pi i)^2}\! \int_{(\sigma_s)} \!\int_{(\sigma_w)}\! \sum_{m \neq 0} \frac{(-1)^m 2^{3s-3w}(1-2^{-4s})  L(4s,\eta^{2m})}{\zeta^{(2)}(8s) \log(1+\sqrt{2})\Gamma(2s-2w)}\\
    & \times
    \Gamma(s-w +it_m)\Gamma(s-w-it_m) U_{\pm x}(s) U_{\pm x}(w)  X^{2s}\, dw ds,
\end{align*}
with $\sigma_s = 10$ and $\sigma_w = \frac{1}{8}$.  Since $L(s,\eta^{2m})$ is entire for $m\neq 0$,
we have $I_{\pm x}^\mathrm{spec}(10,\frac{1}{8},X) = I_{\pm x}^\mathrm{spec}(\frac{1}{8}+\epsilon,\frac{1}{8},X)$.  Our analysis now follows along the lines of~Section~\ref{ssec:discrete_term_ratl_points} to prove that $I_{\pm x}^\mathrm{spec}(\frac{1}{8}+\epsilon,\frac{1}{8},X) = O(X^{\frac{1}{4}+2\epsilon} x^{1+2\epsilon})$.

\begin{proof}[Proof of Theorem~\ref{thm:firsttwoX}]
We have the equality.
\begin{equation}
  S_{\pm x}(X)
  =
  I_{\pm x}^{0}(10, \tfrac{1}{8}, X)
  +
  I_{\pm x}^{\mathrm{spec}}(10,\tfrac{1}{8}, X).
\end{equation}
We have shown that $I_{\pm x}^{0}$ contributes~\eqref{eq:firsttwoX_first_MT},
while $I_{\pm x}^{\mathrm{spec}}$ is $O(X^{\frac{1}{4}+2\epsilon} x^{1+2\epsilon})$.
Together, these show that
\begin{equation}
\begin{split}
  S_{\pm x}(X)
  &=
  \frac{16\sqrt{2}}{\pi^2} \, _2F_1(\tfrac{1}{4},\tfrac{1}{2},\tfrac{5}{4},\tfrac{1}{2})
 X^{\frac{1}{2}} + O\Big(\frac{X^{\frac{1}{2}}}{x}\Big)
  + O_\epsilon(X^{\frac{1}{4} + 2\epsilon}x^{1+2\epsilon}).
\end{split}
\end{equation}
The theorem follows from the choice $x=X^{1/8}$ and the inequalities $S_x^{-}(X) \leq S(X)
\leq S_x^{+}(X)$.
\end{proof}

\vspace{20 mm}
\bibliographystyle{alpha}
\bibliography{compiled_bibliography}

\newcommand{\etalchar}[1]{$^{#1}$}
\begin{thebibliography}{HKLDW19}

\bibitem[Blo17]{blomer2017triple}
Valentin Blomer.
\newblock On triple correlations of divisor functions.
\newblock {\em Bulletin of the London Mathematical Society}, 49(1):10--22,
  2017.

\bibitem[Bom05]{Bo05}
Enrico Bombieri.
\newblock The {R}osetta {S}tone of {$L$}-functions.
\newblock In {\em Perspectives in analysis}, volume~27 of {\em Math. Phys.
  Stud.}, pages 1--15. Springer, Berlin, 2005.

\bibitem[Bum97]{Bump98}
Daniel Bump.
\newblock {\em Automorphic forms and representations}, volume~55 of {\em
  Cambridge Studies in Advanced Mathematics}.
\newblock Cambridge University Press, Cambridge, 1997.

\bibitem[Dic13]{dickson2013history}
Leonard~Eugene Dickson.
\newblock {\em History of the theory of numbers: Diophantine Analysis},
  volume~2.
\newblock Courier Corporation, 2013.

\bibitem[Duk03]{duke2003rational}
W~Duke.
\newblock Rational points on the sphere.
\newblock In {\em Number Theory and Modular Forms}, pages 235--239. Springer,
  2003.

\bibitem[GH85]{GoldfeldHoffstein85}
D.~Goldfeld and J.~Hoffstein.
\newblock Eisenstein series of $\frac{1}{2}$-integral weight and the mean value
  of real {D}irichlet {$L$}-series.
\newblock {\em Inventiones Mathematicae}, \textbf{80}:185--208, 1985.

\bibitem[Gol15]{Goldfeld06}
Dorian Goldfeld.
\newblock {\em Automorphic forms and {L}-functions for the group {${\rm
  GL}(n,\rm R)$}}, volume~99 of {\em Cambridge Studies in Advanced
  Mathematics}.
\newblock Cambridge University Press, Cambridge, 2015.
\newblock With an appendix by Kevin A. Broughan, Paperback edition of the 2006
  original [ MR2254662].

\bibitem[GR15]{GradshteynRyzhik07}
I.~S. Gradshteyn and I.~M. Ryzhik.
\newblock {\em Table of integrals, series, and products}.
\newblock Elsevier/Academic Press, Amsterdam, eighth edition, 2015.
\newblock Translated from the Russian, Translation edited and with a preface by
  Daniel Zwillinger and Victor Moll, Revised from the seventh edition
  [MR2360010].

\bibitem[HH16]{HoffsteinHulse13}
Jeff Hoffstein and Thomas~A. Hulse.
\newblock Multiple {D}irichlet series and shifted convolutions.
\newblock {\em J. Number Theory}, 161:457--533, 2016.
\newblock With an appendix by Andre Reznikov.

\bibitem[HKLDW18]{HulseGaussSphere}
Thomas~A. Hulse, Chan~Ieong Kuan, David Lowry-Duda, and Alexander Walker.
\newblock Second moments in the generalized {G}auss circle problem.
\newblock {\em Forum of Mathematics, Sigma}, 2018.
\newblock Accepted, In Press; arXiv:1703.10347.

\bibitem[HKLDW19]{HKLDWtriplecusp}
Thomas Hulse, Chan~Ieong Kuan, David Lowry-Duda, and Alexander Walker.
\newblock Triple correlation sums of coefficients of cusp forms, 2019.

\bibitem[Iwa97]{Iwaniec97}
Henryk Iwaniec.
\newblock {\em Topics in classical automorphic forms}, volume~17 of {\em
  Graduate Studies in Mathematics}.
\newblock American Mathematical Society, Providence, RI, 1997.

\bibitem[JM05]{jutila2005uniform}
Matti Jutila and Yoichi Motohashi.
\newblock Uniform bound for {H}ecke {$L$}-functions.
\newblock {\em Acta Math.}, 195:61--115, 2005.

\bibitem[Kob93]{Koblitz}
Neal Koblitz.
\newblock {\em Introduction to elliptic curves and modular forms}, volume~97 of
  {\em Graduate Texts in Mathematics}.
\newblock Springer-Verlag, New York, second edition, 1993.
\newblock http://dx.doi.org/10.1007/978-1-4612-0909-6.

\bibitem[KS02]{KS02}
Henry~H. Kim and Freydoon Shahidi.
\newblock Cuspidality of symmetric powers with applications.
\newblock {\em Duke Math. J.}, 112(1):177--197, 2002.

\bibitem[LD17]{LowryDudaThesis}
David Lowry-Duda.
\newblock {\em On Some Variants of the {G}auss Circle Problem}.
\newblock PhD thesis, Brown University, 5 2017.
\newblock \url{https://arxiv.org/abs/1704.02376}.

\bibitem[LY02]{LY02}
Jianya Liu and Yangbo Ye.
\newblock Subconvexity for {R}ankin-{S}elberg {$L$}-functions of {M}aass forms.
\newblock {\em Geom. Funct. Anal.}, 12(6):1296--1323, 2002.

\bibitem[Mic07]{michel2007}
Philippe Michel.
\newblock Analytic number theory and families of automorphic {$L$}-functions.
\newblock In {\em Automorphic forms and applications}, volume~12 of {\em
  IAS/Park City Math. Ser.}, pages 181--295. Amer. Math. Soc., Providence, RI,
  2007.

\bibitem[Nel19]{Nelson19}
Paul~D. Nelson.
\newblock Subconvex equidistribution of cusp forms: reduction to {E}isenstein
  observables.
\newblock {\em Duke Math. J.}, 168(9):1665--1722, 2019.

\bibitem[Nel21]{Nelson21}
Paul~D. Nelson.
\newblock The spectral decomposition of {$|\theta|^2$}.
\newblock {\em Math. Z.}, 298(3-4):1425--1447, 2021.

\bibitem[S{\etalchar{+}}85]{shimura1985eisenstein}
Goro Shimura et~al.
\newblock On eisenstein series of half-integral weight.
\newblock {\em Duke Mathematical Journal}, 52(2):281--314, 1985.

\bibitem[Sar01]{sarnak2001estimates}
Peter Sarnak.
\newblock Estimates for rankin--selberg l-functions and quantum unique
  ergodicity.
\newblock {\em Journal of Functional Analysis}, 184(2):419--453, 2001.

\bibitem[Shi73]{Shimura}
Goro Shimura.
\newblock On modular forms of half integral weight.
\newblock {\em The Annals of Mathematics}, 97(3):440--481, 1973.

\bibitem[Sie80]{Siegel80}
C.L. Siegel.
\newblock {\em Advanced Analytic Number Theory}.
\newblock Studies in mathematics / Tata institute of fundamental research. Tata
  Inst. of Fundamental Research, 1980.

\bibitem[TB18]{TB18}
Ramin Takloo-Bighash.
\newblock {\em A {P}ythagorean introduction to number theory}.
\newblock Undergraduate Texts in Mathematics. Springer, Cham, 2018.
\newblock Right triangles, sums of squares, and arithmetic.

\bibitem[Wu19]{Wu19}
Han Wu.
\newblock Burgess-like subconvexity for $\text{GL}_{1}$.
\newblock {\em Compositio Mathematica}, 155(8):1457–1499, 2019.

\end{thebibliography}

\end{document}